\documentclass{article}
\usepackage{amsmath,amsthm,amssymb,amscd}
\usepackage{color}

\theoremstyle{definition}
\newtheorem{Thm}{Theorem}[section]
\newtheorem{Def}[Thm]{Definition}
\newtheorem{Lem}[Thm]{Lemma}
\newtheorem{Cor}[Thm]{Corollary}
\newtheorem{Ex}[Thm]{Example}
\newtheorem{Prop}[Thm]{Proposition}
\newtheorem{Rem}[Thm]{Remark}
\newtheorem{Quest}[Thm]{Question}

\def\Homeo{\mathrm{Homeo}}
\def\SL{\mathrm{SL}}
\def\Z{\mathbb{Z}}
\def\N{\mathbb{N}}

\def\H{\mathfrak{H}}
\def\K{\mathfrak{K}}

\def\Out{\mathrm{Out}}
\def\bbk{\mathbb{K}}
\def\calk{\mathcal{K}}
\def\cale{\mathcal{E}}

\title{Blowing up and down compacta with geometrically finite convergence actions of a group}
\date{}
\author{
Yoshifumi Matsuda\footnote
{Graduate School of Mathematical Sciences, 
University of Tokyo, 
3-8-1 Komaba, 
Meguro-ku, 
Tokyo, 
153-8914 Japan,
ymatsuda@ms.u-tokyo.ac.jp} \footnote{
The author is partially supported by Grant-in-Aid for Scientific Researches for Young Scientists (B) 
(No. 22740034 and No. 25800036), Japan Society of Promotion of Science.}, 
Shin-ichi Oguni\footnote
{Department of Mathematics, Faculty of Science, Ehime University,
2-5 Bunkyo-cho, 
Matsuyama, 
Ehime, 
790-8577 Japan,
oguni@math.sci.ehime-u.ac.jp} \footnote{
The author is partially supported by Grant-in-Aid for Scientific Researches for Young Scientists (B) 
(No. 24740045), Japan Society of Promotion of Science.}, 
Saeko Yamagata\footnote
{Faculty of Education and Human Sciences, Yokohama National University,
240-8501 Yokohama, Japan,
yamagata@ynu.ac.jp}
}

\begin{document}

\maketitle
\begin{abstract}     
We consider two compacta with minimal non-elementary convergence actions of a countable group. 
When there exists an equivariant continuous map 
from one to the other, we call the first a blow-up of the second 
and the second a blow-down of the first. 
When both actions are geometrically finite, 
it is shown that one is a blow-up of the other if and only if 
each parabolic subgroup with respect to the first is parabolic with respect to the second.  
As an application, 
for each compactum with a geometrically finite convergence action, 
we construct its blow-downs with convergence actions which are not geometrically finite.\\


\noindent
2010MSC:
20F67;
20F65
\end{abstract}

\section{Introduction}
The notion of geometrically finite convergence groups 
is a generalized notion of geometrically finite Kleinian groups 
and is deeply related to the notion of 
relatively hyperbolic groups (see \cite[Definition 1]{Bow12} and \cite[Theorem 0.1]{Yam04}). 
The study of convergence groups was initiated in \cite{G-M87}. 
Also relatively hyperbolic groups have been studied since they were introduced 
in \cite{Gro87} (see for example \cite{Bow12}, \cite{D-S05}, \cite{Far98} and \cite{Osi06a}). 
In this paper we study blow-ups and blow-downs of compact metrizable spaces endowed with
geometrically finite convergence actions of a group
by using relative hyperbolicity of the group.
We note that the actions on such blow-ups and blow-downs are not necessarily 
geometrically finite. 

Let $G$ be a countable group, which has compact metrizable spaces 
$X$ and $Y$ endowed with minimal non-elementary convergence actions of $G$.
Throughout this paper, every countable group is endowed with the discrete topology.
We denote the set of maximal parabolic subgroups 
with respect to $X$ by $\H(X)$ and call $\H(X)$ the {\it peripheral structure} of $X$.
When the action on $X$ is geometrically finite, 
$\H(X)$ is called a {\it proper relatively hyperbolic structure} on $G$ 
(\cite[Definition 2.1]{M-O-Y3}). 
If $X$ and $Y$ are $G$-equivariant homeomorphic, then we have $\H(X)=\H(Y)$. 
Also if we have a $G$-equivariant continuous map from $X$ to $Y$, 
then we have $\H(X)\to \H(Y)$ (Lemma \ref{equivcont}(5)). 
Here $\H(X)\to \H(Y)$ means that 
each maximal parabolic subgroup with respect to $X$ 
is a parabolic subgroup with respect to $Y$ 
(see also the paragraph just before Lemma \ref{order}).
When there exists a $G$-equivariant continuous map from $X$ to $Y$, 
we call $X$ a {\it blow-up} of $Y$ and $Y$ a {\it blow-down} of $X$.

Suppose that the actions on $X$ and $Y$ are geometrically finite. 
If $\H(X)=\H(Y)$, then $X$ and $Y$ are $G$-equivariant homeomorphic
(see \cite[Theorem 0.1]{Yam04} and also \cite[Theorem 5.2]{Hru10}). 
Now we consider $\H(X)\to \H(Y)$ instead of $\H(X)=\H(Y)$.
The following is our main theorem
(compare with the second question in \cite[Introduction]{Mj09}): 
\begin{Thm}\label{geomorder}
Let $G$ be a countable group
and $X$ and $Y$ be compact metrizable spaces endowed 
with geometrically finite convergence actions of $G$.
If $\H(X)\to \H(Y)$, then we have a $G$-equivariant continuous map $\pi:X\to Y$. 
\end{Thm}

\begin{Rem}\label{Folyd'}
For the case where both actions are geometrically finite Kleinian group actions 
which may be higher dimensional, 
Theorem \ref{geomorder} is known by \cite[Corollary of Theorem 1]{Tuk88}. 
The argument is based on existence of the so-called Floyd maps by \cite{Flo80} and \cite[Theorem 1]{Tuk88}.  
By the same way, 
it is possible to give a proof of Theorem \ref{geomorder}
for the case where $G$ is finitely generated. 
Indeed existence of Floyd maps (\cite[Corollary in Section 1.5]{Ger10}) 
and \cite[Theorem A]{G-P09} (or Lemma \ref{equivcont} (4))
imply Theorem \ref{geomorder} for the case where $G$ is finitely generated
(see also \cite[Lemmas 4.13 and 4.14]{Yan11}). 
We remark that geometrically finite convergence actions in this paper 
(see Section 2) and those used in \cite{Ger10} and \cite{G-P09} 
are equivalent when we consider countable groups (see \cite[Section 3]{Ger09}). 

Our approach to Theorem \ref{geomorder} is different from the above
and admits the case of countable groups which are not necessarily finitely generated. 
We note that any geometrically finite Kleinian group is finitely generated, 
but a countable group admitting a geometrically finite convergence action
is not necessarily finitely generated.
\end{Rem}

\begin{Rem}\label{Folyd''}
After earlier versions of this paper were circulated, Gerasimov and Potyagailo 
generalized Theorem \ref{geomorder} for the case where 
$G$ is not necessarily countable and $X, Y$ are not necessarily
metrizable (\cite[Theorem C]{G-P13}).  
\end{Rem}

\begin{Rem}\label{geomorder'}
We consider $\pi:X\to Y$ in Theorem \ref{geomorder}. 
General properties of equivariant continuous maps between 
minimal non-elementary convergence actions 
imply the following (see Lemma \ref{equivcont}): 
\begin{itemize}
\item $\pi$ is surjective;
\item $\pi$ is a unique $G$-equivariant continuous map from $X$ to $Y$;  
\item if $q\in Y$ is a bounded parabolic point and $H$ is its maximal parabolic subgroup, 
then $\pi^{-1}(q)$ is equal to the limit set $\Lambda(H,X)$ of the induced action of $H$ on $X$; 
\item $id_G\cup \pi:G\cup X\to G\cup Y$ is continuous.
\end{itemize}
Here $G\cup X$ and $G\cup Y$ are natural compactifications of $G$ by $X$ and $Y$
(see \cite[Subsection 2.4]{Ger09}, \cite[Proposition 8.3.1]{Ger10} and Lemma \ref{b'}). 
\end{Rem}

By combining some facts with Theorem \ref{geomorder}, 
we construct two families of 
uncountably infinitely many compact metrizable spaces endowed 
with minimal non-elementary convergence actions of $G$ 
which are not geometrically finite as blow-downs. 
The first family consists of uncountably infinitely many blow-downs 
with the same peripheral structure: 
\begin{Cor}\label{CT}
Let $G$ be a countable group and 
let $X$ be a compact metrizable space endowed with a geometrically finite convergence action of $G$. 
Then there exists a family $\{X_\lambda\}_{\lambda\in \{0,1\}^\N}$ of 
compact metrizable spaces endowed with minimal non-elementary convergence actions of $G$ 
which satisfies the following:
\begin{enumerate}
\item[(1)] for every $\lambda\in\{0,1\}^\N$, $\H(X)=\H(X_\lambda)$;
\item[(2)] for every $\lambda\in\{0,1\}^\N$, $X_\lambda$ is a blow-down of $X$;
\item[(3)] for any $\lambda, \lambda'\in\{0,1\}^\N$, 
$X_\lambda$ is a blow-down of $X_{\lambda'}$ if and only if 
$\lambda^{-1}(\{1\})\supset \lambda'^{-1}(\{1\})$;
\item[(4)] for any $\lambda, \lambda'\in\{0,1\}^\N$ such that $\lambda\neq\lambda'$, 
$X_\lambda$ is not $G$-equivariant homeomorphic to $X_{\lambda'}$; 
\item[(5)] for every $\lambda\in\{0,1\}^\N$, 
the action of $G$ on $X_\lambda$ is not geometrically finite,
\end{enumerate}
where $\{0,1\}^\N$ is the set of all functions from $\N$ to $\{0,1\}$.
\end{Cor}
\noindent
The second family consists of uncountably infinitely many blow-downs with  
mutually different peripheral structures: 
\begin{Cor}\label{CT'''}
Let $G$ be a countable group and 
let $X$ be a compact metrizable space endowed 
with a geometrically finite convergence action of $G$. 
Then there exists a family $\{Y_\lambda\}_{\lambda\in \{0,1\}^\N}$ of 
compact metrizable spaces endowed with minimal non-elementary convergence actions of $G$ 
which satisfies the following:
\begin{enumerate}
\item[(1)] for every $\lambda\in\{0,1\}^\N$, $\H(X)\neq \H(Y_\lambda)$;
\item[(2)] for every $\lambda\in\{0,1\}^\N$, $Y_\lambda$ is a blow-down of $X$;
\item[(3)] for any $\lambda, \lambda'\in\{0,1\}^\N$, the following are equivalent:
\begin{itemize}
\item $Y_\lambda$ is a blow-down of $Y_{\lambda'}$;
\item $\H(Y_\lambda)\supset \H(Y_{\lambda'})$;
\item $\lambda^{-1}(\{1\})\supset \lambda'^{-1}(\{1\})$;
\end{itemize}
\item[(4)] for any $\lambda, \lambda'\in\{0,1\}^\N$ such that $\lambda\neq\lambda'$, 
$\H(Y_\lambda)$ is not equal to $\H(Y_{\lambda'})$. In particular, 
$Y_\lambda$ is not $G$-equivariant homeomorphic to $Y_{\lambda'}$; 
\item[(5)] for every $\lambda\in\{0,1\}^\N$, the action of $G$ on $Y_\lambda$ 
is not geometrically finite. In particular, 
$\H(Y_\lambda)$ contains infinitely many conjugacy classes if $\lambda^{-1}(\{1\})$ is infinite.
\end{enumerate}
\end{Cor} 
 
In Section \ref{convact}, we gather and show some properties of convergence actions. 
Also Propositions \ref{noblowup'} and \ref{noblowup''} related to blow-ups of 
compact metrizable spaces with geometrically finite convergence actions are proved. 
In Section \ref{aug}, we recall some terminologies about augmented spaces in \cite{G-M08}. 
In Section \ref{go}, Theorem \ref{geomorder} is proved. 
Also a corollary of Theorem \ref{geomorder} 
for semiconjugacies is shown (Corollary \ref{pphi}).  
In Section \ref{ct},
we deal with applications of Theorem \ref{geomorder}. 
Indeed we prove Corollaries \ref{CT} and \ref{CT'''}. 
Also some other corollaries are given. 
In \ref{no}, a criterion for countable groups to have no non-elementary convergence actions
is shown. 
In \ref{sect-DGO}, we give a remark on hyperbolically embedded virtually free subgroups
of relatively hyperbolic groups.
 
\section{Convergence actions and peripheral structures}\label{convact}
In this section, we recall some definitions related to convergence actions and 
geometrically finite convergence actions of groups 
(refer to \cite{Tuk94}, \cite{Tuk98}, \cite{Bow12} and \cite{Bow99b})
and show some properties. 

Let $G$ be a countable group with a continuous action on a compact metrizable space $X$. 
The action is called a {\it convergence action} of $G$ on $X$
if $X$ has infinitely many points
and for each infinite sequence $\{g_i\}$ of mutually different elements of $G$, 
there exist a subsequence $\{g_{i_j}\}$ of $\{g_i\}$ 
and two points $r,a\in X$ such that $g_{i_j}|_{X\setminus \{r\}}$ converges to $a$ 
uniformly on any compact subset of $X\setminus \{r\}$ 
and also $g_{i_j}^{-1}|_{X\setminus \{a\}}$ converges to $r$ 
uniformly on any compact subset of $X\setminus \{a\}$. 
The sequence $\{g_{i_j}\}$ is called a {\it convergence sequence} 
and also the points $r$ and $a$ are called the {\it repelling point} of $\{g_{i_j}\}$ 
and the {\it attracting point} of $\{g_{i_j}\}$, respectively. 

We fix a convergence action of $G$ on a compact metrizable space $X$.
The set of all repelling points and attracting points is equal to 
the limit set $\Lambda(G,X)$ (\cite[Lemma 2M]{Tuk94}). 
The cardinality of $\Lambda(G,X)$ is $0$, $1$, $2$ or $\infty$ 
(\cite[Theorem 2S, Theorem 2T]{Tuk94}). 
If $\#\Lambda(G,X)=\infty$, then the action of $G$ on $X$ is called a 
{\it non-elementary convergence action}. 
When the action of $G$ on $X$ is a non-elementary convergence action, 
$G$ is infinite and the induced action of $G$ on $\Lambda(G,X)$ 
is a minimal non-elementary convergence action.
We remark that $\#\Lambda(G,X)=0$ if $G$ is finite by definition. 
Also $G$ is virtually infinite cyclic if $\#\Lambda(G,X)=2$ 
(see \cite[Lemma 2Q, Lemma 2N and Theorem 2I]{Tuk94}). 
An element $l$ of $G$ is {\it loxodromic} if it is of infinite order and has exactly two fixed points. 
For a loxodromic element $l\in G$, 
the sequence $\{l^i\}_{i\in \N}$ is a convergence sequence with the repelling point $r$ 
and the attracting point $a$, 
which are distinct and fixed by $l$. 
We call a subgroup $H$ of $G$ a {\it parabolic subgroup} if 
it is infinite, fixes exactly one point and has no loxodromic elements. 
Such a point is called a {\it parabolic point}. 
A parabolic point is {\it bounded} if its maximal parabolic subgroup 
acts cocompactly on its complement. 
We call a point $r$ of $X$ a {\it conical limit point} 
if there exists a convergence sequence $\{g_i\}$ with the attracting point $a\in X$ 
such that the sequence $\{g_i(r)\}$ 
converges to a different point from $a$.  
A convergence action of $G$ on $X$ is {\it geometrically finite} 
if every point of $X$ is either a conical limit point or a bounded parabolic point.
Since $X$ has infinitely many points, 
every geometrically finite convergence action is non-elementary. 
Also since all conical limit points and all bounded parabolic points belong to the limit set, 
every geometrically finite convergence action is minimal.

The following is a special case of \cite[Proposition 8.3.1]{Ger10}, 
which claims that a compact metrizable space endowed with 
a minimal non-elementary convergence action of a countable group 
naturally gives a compactification of the group: 
\begin{Lem}\label{b'}
Let $G$ be a countable group and $X$ be a compact metrizable space
endowed with a minimal non-elementary convergence action of $G$.
Then $X$ gives a unique compactification of $G$ such that 
the natural action of $G$ on $G\cup X$ is a convergence action 
whose limit set is $X$. 
\end{Lem}

Let $G$ be a countable group. 
For a compact metrizable space $X$ endowed with a convergence action of $G$, 
the set $\H(X)$ of all maximal parabolic subgroups is 
conjugacy invariant and {\it almost malnormal}, that is, $H_1\cap H_2$ is finite 
for any different $H_1, H_2\in \H(X)$ 
and every $H\in \H(X)$ is equal to its normalizer in $G$ (\cite[Lemma 3.3]{M-O-Y3}). 
We call $\H(X)$ a {\it peripheral structure} of $X$.  
In general we define a {\it peripheral structure} on $G$ 
as a conjugacy invariant and almost malnormal collection of 
infinite subgroups of $G$. 
For a family $\{\H_i\}_{i\in I}$ of peripheral structures on $G$, 
we define 
\[
\bigwedge_{i\in I}\H_i:=
\{P\subset G\ |\ P\text{ is infinite and }
P=\bigcap_{i\in I}H_i\text{ for some }H_i\in\H_i\text{ for each }i\in I\}, 
\]
which is also a peripheral structure on $G$. 
For two peripheral structures $\H$ and $\K$ on $G$, we put $\K\to \H$ if 
for any $K\in \K$, there exists $H\in \H$ such that $K\subset H$.  
We have the following about $\to$ (\cite[Lemma 3.4]{M-O-Y3}): 
\begin{Lem}\label{order}
Let $G$ be a countable group. 
Then $\to$ is a partial order on the set of peripheral structures on $G$. 
\end{Lem}

We show some basic properties related to equivariant continuous maps 
between compact metrizable spaces endowed with convergence actions
(refer to \cite[Subsection 2.3 and Subsection 2.4]{Ger09} about 
Lemma \ref{equivcont} (1), (2) and (3)). 
\begin{Lem}\label{equivcont}
Let $G$ be a countable group. 
Let $X$ and $Y$ be compact metrizable spaces endowed with 
minimal non-elementary convergence actions of $G$. 
Suppose that there exists a $G$-equivariant continuous map $\pi$ from $X$ to $Y$. 
Then we have the following: 
\begin{enumerate}
\item[(1)] $\pi$ is surjective;
\item[(2)] $\pi$ is a unique $G$-equivariant continuous map from $X$ to $Y$;  
\item[(3)] if $r\in Y$ is a conical limit point, then $\pi^{-1}(r)$ consists of one conical limit point;
\item[(4)] if $q\in Y$ is a bounded parabolic point and $H$ is its maximal parabolic subgroup, 
then $\pi^{-1}(q)=\Lambda(H,X)$; 
\item[(5)] $\H(X)\to \H(Y)$;
\item[(6)] $id_G\cup \pi:G\cup X\to G\cup Y$ is continuous.
\end{enumerate}
\end{Lem}

\begin{Rem}
Lemma \ref{equivcont} (4) is considered as a generalization of \cite[Theorem A]{G-P09}.
\end{Rem}

\begin{proof}[Proof of Lemma \ref{equivcont}]
(1)~ Minimality of the action on $Y$ implies that $\pi$ is surjective. 

(2)~ Suppose that $\pi':X\to Y$ is a $G$-equivariant continuous map. 
We take any $p\in X$. 
There exists a convergence sequence $\{g_i\}$
whose attracting point is $p$. 
Then $\pi'(p)$ as well as $\pi(p)$ is the attracting point of 
the convergence sequence $\{g_i\}$. 
Hence we have $\pi(p)=\pi'(p)$.

(3)~ Suppose that $r\in Y$ is a conical limit point. 
Then we have a convergence sequence $\{g_i\}$ with respect to $Y$ 
such that the repelling point is $r$ and $g_i r$ converges to $b\in Y$ which is 
different from the attracting point $a\in Y$.  
We take a convergence subsequence $\{g_{i_j}\}$ of $\{g_i\}$ with respect to $X$. 
We have the repelling point $r'\in X$ and 
the attracting point $a'\in X$.  
Also we can assume that $g_{i_j}r'$ converges to $b'\in X$. 
Then we have $\pi(r')=r$, $\pi(a')=a$ and $\pi(b')=b$. 
Since $a\neq b$, we have $a'\neq b'$. 
Hence $r'$ is a conical limit point. 
Assume that there exists $z\in \pi^{-1}(r)$ such that $z\neq r'$.
Then we have $g_{i_{j}}z\to a'$ and thus $g_{i_{j}}r\to \pi(a')=a$.
This contradicts the fact that $a\neq b$.

(4)~ Clearly we have $\Lambda(H,X)\subset \pi^{-1}(q)$.
Since $\Lambda(H,X)$ has at least one point, we assume that $\pi^{-1}(q)$ has at least two points.
We prove that $\pi^{-1}(q)\subset \Lambda(H,X)$.
Since $q\in Y$ is a bounded parabolic point, 
we have a compact subset $K\subset Y\setminus \{q\}$
such that $HK=Y\setminus \{q\}$. 
We have $H\pi^{-1}(K)=X\setminus \pi^{-1}(q)$. 
For any $x\in \pi^{-1}(q)$, we have a convergence sequence $\{g_i\}$
whose attracting point is $x$. 
If $\{g_i\}$ have a subsequence whose elements belong to a right coset of $H$, 
then $x\in \Lambda (H,X)$. 
We assume that all elements of $\{g_i\}$ belong to mutually different right cosets of $H$
and does not belong to $H$. 
Then we have some $h'_i\in H$ and $r_i\in G\setminus H$ such that $g_i=h'_ir_i$ for any $i$. 
Since $\pi^{-1}(q)$ has at least two points, we have a point $x'\in \pi^{-1}(q)$
that is not the repelling point of $\{g_i\}$. 
Hence we have $h'_ir_ix'\to x$ and $r_ix'\notin \pi^{-1}(q)$. 
We have $h''_i\in H$ and $k'_i\in \pi^{-1}(K)$ such that $r_ix'=h''_ik'_i$. 
Now we have a sequence $\{h_i:=h'_ih''_i\}$ of $H$ and $\{k'_i\}$ of $\pi^{-1}(K)$
such that $h_ik'_i\to x$ in $X$. 
Since $x\notin H\pi^{-1}(K)$, 
we can assume that $\{h_i\}$ is a sequence of mutually different elements. 
We take a convergence subsequence $\{h_{i_j}\}$ with respect to $X$. 
Both its attracting point $a$ and its repelling point $r$ are in $\pi^{-1}(q)$.
Since $\pi^{-1}(K)$ is compact and does not contain $r$, 
$h_{i_j}(\pi^{-1}(K))$ uniformly converges to $a$. 
Then $h_{i_j}k'_{i_j}$ must converge to $a$. 
Now we have $x=a$. Hence $\pi^{-1}(q)\subset \Lambda(H,X)$. 

(5)~ We prove $\H(X)\to \H(Y)$. 
We take $K\in \H(X)$ with the fixed point $p$. 
Then $\pi(p)$ is fixed by $K$. 
Assume that there exists a loxodromic element $l\in G$ with respect to $Y$
such that $\pi(p)$ is the repelling point of $\{l^n\}_{n\in \N}$.
Since $\pi(p)$ is conical, $\pi^{-1}(\pi(p))=\{p\}$ by Lemma \ref{equivcont} (3). 
Hence $l$ fixes $p$ and thus $l\in K$. 
Hence $l$ has the only fixed point $\pi(p)$. 
This contradicts the assumption that $l$ is loxodromic.
Hence $\pi(p)$ is a parabolic point with a parabolic subgroup $K$.  

(6)~ We take a sequence $\{g_n\}$ which converges to $x\in X$
in $G\cup X$. 
We take any accumulation point $y\in Y$ of $\{g_n\}$ 
and a subsequence $\{g_{n_i}\}$ such that $g_{n_i}\to y$ in $G\cup Y$. 
We take a convergence subsequence $\{g_{n_{i_j}}\}$ 
with respect to both $G\cup X$ and $G\cup Y$. 
The attracting points in $G\cup X$ and $G\cup Y$ are $x$ and $y$, 
respectively. 
Then $x$ and $y$ are the attracting points in $X$ and $Y$, 
respectively. 
Hence we have $\pi(x)=y$. 
Thus we have $g_n\to y$ in $G\cup Y$. 
\end{proof}

Lemma \ref{equivcont} (1), (3) and (4) imply the following
(compare with Corollary \ref{CT}):
\begin{Prop}\label{noblowup'}
Let $G$ be a countable group. 
Let $X$ be a compact metrizable space endowed with 
a minimal non-elementary convergence action of $G$
and $Y$ be a compact metrizable space endowed with 
a geometrically finite convergence action of $G$. 
Let $\pi:X\to Y$ be a $G$-equivariant continuous map. 
If $\H(X)=\H(Y)$, 
then $\pi$ is a $G$-equivariant homeomorphism.
\end{Prop}

The following gives a sufficient condition that 
a compact metrizable space endowed with a geometrically finite convergence action 
has no proper blow-ups: 
\begin{Prop}\label{noblowup''}
Let $G$ be a countable group. 
Let $X$ be a compact metrizable space endowed with 
a geometrically finite convergence action of $G$. 
If every $H\in \H(X)$ is not virtually cyclic 
and has no non-elementary convergence actions, 
then $X$ has no proper blow-ups. 
\end{Prop}

\begin{proof}
Assume that every $H\in \H(X)$ is not virtually cyclic 
and has no non-elementary convergence actions. 
Let $Y$ be a compact metrizable space endowed 
with a non-elementary convergence action of $G$
and $\pi:Y\to X$ be a $G$-equivariant continuous map.
By the assumption, every $H\in \H(X)$ is a parabolic subgroup with
respect to $Y$ and hence we have $\H(X)\to \H(Y)$. On the other hand
Lemma \ref{equivcont} (5) implies that we have $\H(Y) \to \H(X)$. 
We have $\H(Y)=\H(X)$ by Lemma \ref{order}. 
Hence it follows from Proposition \ref{noblowup'} that $\pi$ is a $G$-equivariant homeomorphism. 
\end{proof}

\begin{Ex}
If $G$ is a non-elementary hyperbolic group, then 
the action on the ideal boundary $\partial G$
is a geometrically finite convergence action 
with no parabolic points 
(see \cite[8.2]{Gro87} and also \cite[Theorem 3.4 and Theorem 3.7]{Fre95}). 
Then there exist no proper blow-ups of $\partial G$ (\cite[Subsection 2.5]{Ger09}). 
We recognize other examples by Proposition \ref{noblowup''}.  
Indeed the limit set of any geometrically finite Kleinian group 
such that the set of maximal parabolic subgroups does not contain virtually cyclic subgroups
has no proper blow-ups because maximal parabolic subgroups are virtually abelian
(refer to \cite[Theorem 2U]{Tuk94}). 
Also if a group $G$ is a free product of finitely many groups which are either 
$\Z^k\ (k \ge 2)$, $\SL(n,\Z)\ (n \ge 3)$ or the mapping class group of an orientable surface 
of genus $g$ with $p$ punctures, where $3g+p \ge 5$, then 
the set of all conjugates of factors of $G$ 
is a peripheral structure of 
a compact metrizable space endowed with a geometrically finite convergence action 
which has no proper blow-ups 
(see \ref{no} and refer to \cite[Remark (II) in Section 7]{M-O-Y3}).  
\end{Ex}

\begin{Lem}\label{phi}
Let $G$ be a countable group
and $X$ and $Y$ be compact metrizable spaces endowed 
with geometrically finite convergence actions of $G$.
Let $\phi:G\to G$ be an automorphism of $G$.
Suppose that $\phi(\H(X))= \H(Y)$. 
Then we have a homeomorphism $\Phi:X\to Y$ 
such that $\Phi(gx)=\phi(g)\Phi(x)$ for any $g\in G$ and any $x\in X$.
Moreover $\phi\cup \Phi:G\cup X\to G\cup Y$ is a homeomorphism.
\end{Lem}
\begin{proof}
When we denote the action on $X$ by $\rho:G\to \Homeo (X)$, 
$\rho\circ \phi^{-1}:G\to \Homeo (X)$ is another action on $X$. 
We denote the compact space $X$ with the action $\rho\circ \phi^{-1}$
by $X_\phi$. 
Then $\rho\circ \phi^{-1}$ is a geometrically finite convergence action of $G$
and the set of maximal parabolic subgroups is nothing but $\phi(\H(X))$.
Then $id_X:X\to X_\phi$ satisfies that $id_X(\rho(g)x)=\rho\circ\phi^{-1}(\phi(g))id_X(x)$
for any $g\in G$ and any $x\in X$.
Obviously $\phi\cup id_X:G\cup X\to G\cup X_\phi$ is a homeomorphism.
Since $\phi(\H(X))$ is a proper relatively hyperbolic structure, 
there exists a $G$-equivariant homeomorphism $\Phi':X_\phi\to Y$ 
(see \cite[Theorem 0.1]{Yam04} and also \cite[Theorem 5.2]{Hru10}).
Then $id_G\cup \Phi':G\cup X_\phi\to G\cup Y$ is a $G$-equivariant homeomorphism
(Lemma \ref{equivcont} (6)). 
Hence $\Phi:=\Phi'\circ id_X:X\to Y$ satisfies the conditions. 
\end{proof}

The following definition is based on \cite[Subsection 2.4]{Ger09}:
\begin{Def}\label{diag''}
Let $G$ be a countable group and 
$\{X_i\}_{i\in I}$ be a family of 
compact metrizable spaces endowed with minimal non-elementary convergence actions
of $G$, where $I$ is a non-empty countable index set. 
When we consider the diagonal map 
$\Delta:G\to \prod_{i\in I}(G\cup X_i)$ and the closure $\overline{\Delta (G)}$ 
of $\Delta(G)$ in $\prod_{i\in I}(G\cup X_i)$, 
we denote $\overline{\Delta (G)}\setminus \Delta(G)$ by $\bigwedge_{i\in I} X_i$.
\end{Def}

\begin{Lem}\label{diag}
Let $G$ be a countable group and 
$\{X_i\}_{i\in I}$ be a family of 
compact metrizable spaces endowed with minimal non-elementary convergence actions
of $G$, where $I$ is a non-empty countable index set. 
Let $Z$ be a compact metrizable space endowed with a minimal non-elementary convergence action of $G$
and $\pi_i:Z\to X_i$ be a $G$-equivariant continuous map for every $i\in I$. 
Then the image of $\prod \pi_i:Z\to \prod_{i\in I} X_i$ is $\bigwedge_{i\in I} X_i$.  
Also the action of $G$ on $\bigwedge_{i\in I} X_i$ 
is a minimal non-elementary convergence action such that 
$\H(\bigwedge_{i\in I} X_i)=\bigwedge_{i\in I} \H(X_i)$.
In particular $\bigwedge_{i\in I} X_i$ is the smallest common blow-up of all $X_i$.
\end{Lem}
\begin{proof}
Since $\Delta\cup \prod_{i\in I}\pi_i:G\cup Z\to \prod_{i\in I}(G\cup X_i)$ is continuous
by Lemma \ref{equivcont} (6), 
the image is compact and $\Delta(G)$ is dense in the image. 
Hence the image is equal to $\overline{\Delta(G)}$ and thus the image of 
$Z$ is $\bigwedge_{i\in I} X_i$.

Since $\bigwedge_{i\in I} X_i$ is the image of 
$\prod \pi_i:Z\to \prod_{i\in I} X_i$ and the action of $G$ on $Z$ 
is a minimal non-elementary convergence action, 
the action of $G$ on $\bigwedge_{i\in I} X_i$ is a minimal non-elementary convergence action.
We take $H_i\in \H_i$ for every $i\in I$. 
If $\bigcap_{i\in I}H_i$ is infinite,
then $\{a_i\}_{i\in I}\in \prod_{i\in I}X_i$
is an element of $\bigwedge_{i\in I} X_i$,
where $a_i$ is the parabolic point of $H_i$ for every $i\in I$.
Also $\{a_i\}_{i\in I}$ is fixed by $\bigcap_{i\in I}H_i$.
If $\{a_i\}_{i\in I}$ is fixed by $g\in G$ of infinite order, 
then $a_i$ is fixed by $g$ for every $i\in I$
and thus $g$ belongs to $H_i$ for each $i\in I$. 
Hence $\{a_i\}_{i\in I}$ is a unique fixed point of $g$ and thus 
$g$ is not loxodromic with respect to $\bigwedge_{i\in I} X_i$.
Thus $\bigcap_{i\in I}H_i$ is a parabolic subgroup with respect to $\bigwedge_{i\in I} X_i$. 
In particular we have $\bigwedge_{i\in I} \H(X_i)\to \H(\bigwedge_{i\in I} X_i)$. 
Next we prove $\H(\bigwedge_{i\in I} X_i)\to \bigwedge_{i\in I} \H(X_i)$. 
We take a parabolic point $\{a_i\}_{i\in I}\in\bigwedge_{i\in I} X_i$ 
with the maximal parabolic subgroup $L\in \H(\bigwedge_{i\in I} X_i)$. 
Then $a_i$ is fixed by $L$ for every $i\in I$. 
We assume that there exists a loxodromic element $l\in G$ 
with respect to $X_i$ fixing $a_i$ for some $i\in I$. 
Since $a_i$ is conical, $\{a_i\}_{i\in I}$ is conical 
by Lemma \ref{equivcont} (3). 
This contradicts the fact that any conical limit point is not parabolic 
(\cite[Theorem 3A]{Tuk98}). 
Hence $L$ is a parabolic subgroup with respect to $X_i$ for every $i\in I$. 
Now we have $\H(\bigwedge_{i\in I} X_i)\to \bigwedge_{i\in I} \H(X_i)$ 
and thus $\H(\bigwedge_{i\in I} X_i)= \bigwedge_{i\in I} \H(X_i)$ 
by Lemma \ref{order}.
\end{proof}

\begin{Lem}\label{inverse}
Let $G$ be a countable group. 
Let $(\{X_i\}_{i\in I}, \{\phi_{ij}:X_j\to X_i\}_{j>i})$ 
be a projective system of 
compact metrizable spaces endowed with 
minimal non-elementary convergence actions of $G$
and $G$-equivariant continuous surjections, where $I$ is a non-empty countable directed set. 
Then the projective limit $\underleftarrow{\lim}_{i\in I}X_i$ is a compact metrizable space
with a minimal non-elementary convergence action of $G$. 

Moreover $\underleftarrow{\lim}_{i\in I}X_i$ and $\bigwedge_{i\in I} X_i$ are $G$-equivariant homeomorphic. 

Also if $r=\{r_i\}_{i\in I}\in \underleftarrow{\lim}_{i\in I}X_i$ is a conical limit point, 
then there exists $i\in I$ such that $r_i\in X_i$ is a conical limit point. 
\end{Lem}
\begin{proof}
Since $I$ is countable, $\underleftarrow{\lim}_{i\in I}X_i$ is a compact metrizable space.
Then the first part is a special case of \cite[Corollary 1 of Proposition P]{Ger09}. 

By universality of the projective limit, we have a $G$-equivariant continuous map
from $\bigwedge_{i\in I} X_i$ to $\underleftarrow{\lim}_{i\in I}X_i$. 
On the other hand we have a $G$-equivariant continuous map
from  $\underleftarrow{\lim}_{i\in I}X_i$ to $\bigwedge_{i\in I} X_i$ by Lemma \ref{diag}. 
Then $\underleftarrow{\lim}_{i\in I}X_i$ and $\bigwedge_{i\in I} X_i$ are $G$-equivariant homeomorphic
by Lemma \ref{equivcont} (2).

Suppose that $r=\{r_i\}_{i\in I}\in \underleftarrow{\lim}_{i\in I}X_i$ is a conical limit point.
Then we have a point $b=\{b_i\}_{i\in I}\in \underleftarrow{\lim}_{i\in I}X_i$ and 
a convergence sequence $\{g_n\}$ with the attracting point 
$a=\{a_i\}_{i\in I}\in \underleftarrow{\lim}_{i\in I}X_i$
such that $r$ is the repelling point, $g_n(r)$ converges to $b$ and $b$ is not $a$.
There exists $i_0\in I$ such that $b_{i_0}\neq a_{i_0}$. 
Hence $r_{i_0}$ is a conical limit point. 
\end{proof}

\section{Relative hyperbolicity and augmented spaces}\label{aug}
In this section, we recall some terminologies about augmented spaces in \cite{G-M08}
and show some lemmas.

Let $G$ be a countable group and $\mathbb H$ be a finite family of proper infinite subgroups of $G$. 
We denote the family of all left cosets by $\alpha:=\bigcup_{H\in{\mathbb H}}G/H$.
We take a left invariant proper metric $d_G$ on $G$.  
We recall the definition of
the {\it augmented space} $X(G,{\mathbb H},d_G)$ 
(see \cite[Section 3]{G-M08} and also \cite[Section 4]{Hru10}).
We put $\Z_{\ge 0}:=\N\cup\{0\}$.
Its vertex set is 
$G\cup (G\times\{0\})\cup \bigcup_{A\in\alpha}(A\times \Z_{\ge 0})/\sim$, where 
$G\ni g\sim(g,0)\in G\times\{0\}$ for any $g\in G$ and 
$G\times \{0\}\ni(a,0)\sim(a,0)\in A\times \{0\}$ for any $A\in\alpha$ and $a\in A$.
Its edge is either a vertical edge or a horizontal edge: 
a {\it vertical edge} is a pair $\{(a,t_1),(a,t_2)\}\subset A\times \Z_{\ge 0}$
such that $|t_1-t_2|=1$ for $A\in \alpha$; 
a {\it horizontal edge} is a pair $\{(a_1,t),(a_2,t)\}\subset A\times \Z_{\ge 0}$
such that $0<d_G(a_1,a_2)\le 2^t$ for $A\in \alpha$
or a pair $\{g_1,g_2\}\subset G$ such that $0<d_G(g_1,g_2) \le 1$.  
We remark that $G$ is finitely generated relative to $\mathbb H$ if and only if 
there exists a left invariant proper metric $d_G$ on $G$ such that 
the augmented space $X(G,{\mathbb H},d_G)$ is connected.
We suppose that the augmented space $X(G,{\mathbb H},d_G)$ is connected.
Then it has the graph metric $d_{X(G,{\mathbb H},d_G)}$. 
Also $S:=\{g\in G \ |\ 0<d_G(e,g)\le 1\}$ is a {\it finite relatively generating set}, 
that is,  $S$ is a finite subset of $G$ and $G$ is generated by $S\cup\bigcup_{H\in \mathbb H} H$. 
We consider the {\it relative Cayley graph} $\overline{\Gamma}(G,{\mathbb H},S)$. 
Its vertex set is $G$. Its edge is a pair $\{g_1,g_2\}\subset G$
such that $0<d_G(g_1,g_2) \le 1$ or 
a pair $\{a_1,a_2\}\subset A$ such that $a_1\neq a_2$ for $A\in \alpha$. 
Then the relative Cayley graph $\overline{\Gamma}(G,{\mathbb H},S)$ is connected
and has the graph metric. 
When we consider 
a left invariant proper metric $d_H:=d_G|_{H\times H}$ on each $H\in{\mathbb H}$, 
$X(H,\{H\},d_H)$ is identified with the full
subgraph of $H\times \Z_{\ge 0}$ in $X(G,{\mathbb H},d_G)$. 
For any $A=gH\in \alpha$, 
the full subgraph of $A\times \Z_{\ge 0}$ in $X(G,{\mathbb H},d_G)$ is 
$gX(H,\{H\},d_H)\subset X(G,{\mathbb H},d_G)$, 
which is called a {\it combinatorial horoball} on $A$ in $X(G,{\mathbb H},d_G)$.
We take $A=gH\in \alpha$ and $g_1, g_2\in A$ such that $g_1\neq g_2$. 
We consider a path $\gamma$ in 
$X(G,{\mathbb H},d_G)$ from $g_1$ to $g_2$
that is contained in the combinatorial horoball 
$gX(H,\{H\},d_G)$ as follows: 
if $0<d_G(g_1,g_2)\le 1$, then $\gamma$ is a path of one edge from 
$(g_1,0)\in gH\times \Z_{\ge 0}$ to $(g_2,0)\in gH\times \Z_{\ge 0}$; 
if $2^{n-1}<d_G(g_1,g_2)\le2^{n}$ for some $n\in \N$, then $\gamma$
runs on vertices of $gH\times \Z_{\ge 0}$
\begin{align*}
&(g_1,0), (g_1,1), (g_1,2), \cdots, \\
&(g_1,n-1), (g_1,n), (g_2,n),(g_2,n-1), \\ 
&\cdots,(g_2,2), (g_2,1) , (g_2,0)
\end{align*}
in order. We call such a path the {\it $H$-typical quasigeodesic} from $g_1$ to $g_2$.
Indeed we can easily confirm that 
there exist constants $\mu_0\ge 1$ and $C_0\ge 0$ 
such that every $H$-typical quasigeodesic is a 
$(\mu_0, C_0)$-quasigeodesic in $X(G,{\mathbb H},d_G)$ for every $H\in \mathbb H$.

We summarize some known facts as follows 
(see for example \cite[Sections 3, 4, 5]{Hru10}): 
\begin{Thm}\label{s}
Let $G$ be a countable group and $\mathbb H$ be a finite family of proper infinite subgroups of $G$.    
Then $G$ has a compact metrizable space $X$ endowed 
with a geometrically finite convergence action of $G$ 
such that $\mathbb H$ is a set of representatives of conjugacy classes of $\H(X)$
if and only if $G$ is not virtually cyclic and there exists a left invariant proper metric $d_G$ on $G$
such that $X(G,{\mathbb H},d_G)$ is connected and hyperbolic. 
When $X(G,{\mathbb H},d_G)$ is connected and hyperbolic 
for a left invariant proper metric $d_G$ on $G$, 
the action of $G$ on the ideal boundary $\partial X(G,{\mathbb H},d_G)$ 
is a geometrically finite convergence action such that $\mathbb H$ 
is a set of representatives of conjugacy classes of $\H(\partial X(G,{\mathbb H},d_G))$
and also the action of $G$ on $X(G,{\mathbb H},d_G)\cup \partial X(G,{\mathbb H},d_G)$ 
is a convergence action
whose limit set is $\partial X(G,{\mathbb H},d_G)$.
\end{Thm}

Let $G$ be a countable group and $\mathbb H$ be a finite family of proper infinite subgroups of $G$. 
Suppose that $G$ is not virtually cyclic and has a left invariant proper metric $d_G$ on $G$
such that $X(G,{\mathbb H},d_G)$ is connected and hyperbolic. 
We consider the following maps:
\begin{align*}
i_{\mathbb H}&:G\to X(G,{\mathbb H},d_G);\\
\pi_{\mathbb H}&:X(G,{\mathbb H},d_G)\to \overline{\Gamma}(G,{\mathbb H},S),
\end{align*}
where $i_{\mathbb H}$ is the natural embedding and 
$\pi_{\mathbb H}$ is the natural surjection.
Then $\pi_{\mathbb H}\circ i_{\mathbb H}$ is also the natural embedding.

\begin{Lem}\label{b}
Let $G$ be a countable group and $\mathbb H$ be a finite family of proper infinite subgroups of $G$. 
Suppose that $G$ is not virtually cyclic and has a left invariant proper metric $d_G$ on $G$
such that $X(G,{\mathbb H},d_G)$ is connected and hyperbolic. 
We consider the subset 
$i_{\mathbb H}(G)\cup \partial X(G,{\mathbb H},d_G)\subset X(G,{\mathbb H},d_G)\cup\partial X(G,{\mathbb H},d_G)$.
Then the action of $G$ on $i_{\mathbb H}(G)\cup \partial X(G,{\mathbb H},d_G)$ 
is a convergence action
whose limit set is $\partial X(G,{\mathbb H},d_G)$. 
\end{Lem}
\begin{proof}
We prove that 
any point of $\partial X(G,{\mathbb H},d_G)$ is an accumulation point of $i_{{\mathbb H}}(G)$ 
in the compact space $X(G,{\mathbb H},d_G)\cup \partial X(G,{\mathbb H},d_G)$. 
Since the action of $G$ on $X(G,{\mathbb H},d_G)$ is proper, 
$i_{{\mathbb H}}(G)$ is unbounded in $X(G,{\mathbb H},d_G)$.  
Hence we have a point $p$ of $\partial X(G,{\mathbb H},d_G)$ 
that is an accumulation point of $i_{{\mathbb H}}(G)$. 
Since the action of $G$ on $\partial X(G,{\mathbb H},d_G)$ is minimal, 
every point of $\partial X(G,{\mathbb H},d_G)$ is an accumulation point of $i_{{\mathbb H}}(G)$.
\end{proof}
\begin{Rem}\label{b''}
We note that $i_{\mathbb H}(G)\cup \partial X(G,{\mathbb H},S)$ induces 
a compactification $G\cup \partial X(G,{\mathbb H},S)$ of $G$. 
The notations $G\cup \partial X(G,{\mathbb H},S)$ on the above 
and $G\cup \partial X(G,{\mathbb H},S)$ in Lemma \ref{b'}
are compatible by Lemma \ref{b'}. 
\end{Rem}

We identify 
$i_{\mathbb H}(G)$ and $\pi_{\mathbb H}\circ i_{\mathbb H}(G)$ with $G$
and also identify 
$i_{\mathbb H}(g)$, $\pi_{\mathbb H}\circ i_{\mathbb H}(g)$ with $g$ 
for every $g\in G$.
For any locally minimal path $\gamma$ without backtracking in $\overline{\Gamma}(G,{\mathbb H},S)$, 
by replacing every $H$-component of $\gamma$ with the $H$-typical quasigeodesic 
for every $H\in \mathbb H$, 
we have a path $\gamma'$ in $X(G,{\mathbb H},d_G)$, which we call the {\it typical lift} of $\gamma$. 
See \cite{Osi06a} about locally minimal paths without backtracking.
The following is a special case of \cite[Lemma 7.3]{Bow12} (see also \cite[Lemma 6.8]{Hru10}):
\begin{Lem}\label{d}
Let $G$ be a countable group and $\mathbb H$ be a finite family of proper infinite subgroups of $G$. 
Suppose that $G$ is not virtually cyclic and has a left invariant proper metric $d_G$ on $G$
such that $X(G,{\mathbb H},d_G)$ is connected and hyperbolic. 
For any constants $\mu\ge 1$ and $C\ge 0$, 
there exist constants $\mu'\ge 1$ and $C'\ge 0$ 
such that the typical lift of any locally minimal $(\mu,C)$-quasigeodesic without backtracking 
in $\overline{\Gamma}(G,{\mathbb H},S)$ is a $(\mu',C')$-quasigeodesic in 
$X(G,{\mathbb H},d_G)$. 
\end{Lem}

\begin{Lem}\label{e}
Let $G$ be a countable group and $\mathbb H$ be a finite family of proper infinite subgroups of $G$. 
Suppose that $G$ is not virtually cyclic and has a left invariant proper metric $d_G$ on $G$ 
such that $X(G,{\mathbb H},d_G)$ is connected and hyperbolic. 
Then for each pair of constants $(\mu,C)$ with $\mu \ge 1$ and $C \ge 0$, 
there exists a non-decreasing function $f_1=f_1^{(\mu,C)}:[0,\infty)\to [0,\infty)$ 
such that $f_1(N)\to \infty$ as $N\to \infty$ and $f_1$ satisfies the following: 
for any $A=gH\in \alpha$, any $g_1,g_2\in A$ such that $g_1\neq g_2$, 
any $(\mu,C)$-quasigeodesic $\gamma$ in $X(G,{\mathbb H},d_G)$ from $g_1$ to $g_2$, 
and any $N>0$, if $d_{X(G,{\mathbb H},d_G)}(e,g_i)\ge N$ for $i=1,2$, 
then $d_{X(G,{\mathbb H},d_G)}(e,\gamma)\ge f_1(N)$. 
\end{Lem}
\begin{proof}
There exist constants $\mu_0\ge 1$ and $C_0\ge 0$ such that 
any $H$-typical quasigeodesic is a $(\mu_0,C_0)$-quasigeodesic in $X(G,{\mathbb H},d_G)$ 
for every $H\in \mathbb H$.

We put $d_H:=d_G|_{H\times H}$. 
There exists a constant $c_1\ge 0$ such that 
$gX(H,\{H\},d_H)$ is a $c_1$-quasiconvex subgraph 
of $X(G,{\mathbb H},d_G)$ for every $A=gH\in \alpha$ 
because $X(G,{\mathbb H},d_G)$ is hyperbolic 
and we can easily confirm that there exist constants $\mu'_1\ge 1, c'_1\ge 0$ such that 
$gX(H,\{H\},d_H)$ is a $(\mu'_1, c'_1)$-quasi-isometrically embedded subgraph 
of $X(G,{\mathbb H},d_G)$ for every $A=gH\in \alpha$. 

We fix a vertex $(a, 0)\in A\times \{0\}$ that is one of the nearest vertices from $e$
for every $A\in \alpha$. 
Then there exists a constant $c_2\ge 0$ such that 
any geodesic from $e$ to any vertex in $gX(H,\{H\},d_H)$
intersects the $c_2$-neighborhood of $(a,0)$ in $X(G,{\mathbb H},d_G)$
for any $A=gH\in \alpha$
because $X(G,{\mathbb H},d_G)$ is hyperbolic and 
$gX(H,\{H\},d_H)$ is $c_1$-quasiconvex for any $A=gH\in \alpha$.

We take $A=gH\in \alpha$, $g_1,g_2\in A$ such that $g_1\neq g_2$. 
We denote the $H$-typical quasigeodesic from $g_1$ to $g_2$ by $\gamma_0$. 
Now we claim that 
\begin{eqnarray*}
d_{X(G,{\mathbb H},d_G)}(e,v) \ge 
\frac{1}{\mu_0}\min_{i=1,2}\frac{1}{2}d_{X(G,{\mathbb H},d_G)}(a,g_i)-\frac{C_0}{\mu_0}-2c_2
\end{eqnarray*}
for any vertex $v$ of $\gamma_0$. 
Indeed we have $v=(g_i,m)$ for some $i \in \{1,2\}$ and some $m \ge 0$. 
We fix $i$ on the above and denote the $H$-typical quasigeodesic from $a$ to $g_i$ by $\gamma_i$. 
We put 
\begin{eqnarray*}
n_i:=\max\{n ~|~ (g_i,n)\text{ is a vertex of }\gamma_i\}. 
\end{eqnarray*}
For each $0 \le k \le n_i$, we denote the length of the segment of $\gamma_i$ with endpoints $a$ 
and $(g_i,k)$ by $l(k)$. 
Since $\gamma_i$ is a $(\mu_0,C_0)$-quasigeodesic, we have 
\begin{eqnarray*}
\mu_0d_{X(G,{\mathbb H},d_G)}(a,(g_i,k))+C_0 \ge l(k). 
\end{eqnarray*}
Also we have 
\begin{eqnarray*}
l(k) \ge \frac{1}{2}d_{X(G,{\mathbb H},d_G)}(a,g_i). 
\end{eqnarray*}
Hence if $m \le n_i$, then 
\begin{eqnarray*}
\mu_0d_{X(G,{\mathbb H},d_G)}(a,(g_i,m))+C_0 \ge \frac{1}{2}d_{X(G,{\mathbb H},d_G)}(a,g_i). 
\end{eqnarray*}
Also if $m \ge n_i$, then 
$d_{X(G,{\mathbb H},d_G)}(a,(g_i,m)) \ge d_{X(G,{\mathbb H},d_G)}(a,(g_i,n_i))$ 
and hence 
\begin{eqnarray*}
\mu_0d_{X(G,{\mathbb H},d_G)}(a,(g_i,m))+C_0 & \ge & \mu_0d_{X(G,{\mathbb H},d_G)}(a,(g_i,n_i))+C_0 \\
& \ge & \frac{1}{2}d_{X(G,{\mathbb H},d_G)}(a,g_i). 
\end{eqnarray*}
Thus we have 
\begin{eqnarray*}
\mu_0d_{X(G,{\mathbb H},d_G)}(a,v)+C_0 \ge \frac{1}{2}d_{X(G,{\mathbb H},d_G)}(a,g_i). 
\end{eqnarray*}
Since it follows from the choice of $c_2$ that 
\begin{eqnarray*}
d_{X(G,{\mathbb H},d_G)}(e,v) & \ge & d_{X(G,{\mathbb H},d_G)}(e,a)+d_{X(G,{\mathbb H},d_G)}(a,v)-2c_2 \\
& \ge & d_{X(G,{\mathbb H},d_G)}(a,v)-2c_2, 
\end{eqnarray*}
the claim follows. 

We take $N >0$ and suppose that $d_{X(G,{\mathbb H},d_G)}(e,g_i)\ge N$ for $i=1,2$. 
If $d_{X(G,{\mathbb H},d_G)}(e,a) \ge \frac{2N}{3}$, then it follows from the choice of $a$ that 
\begin{eqnarray*}
d_{X(G,{\mathbb H},d_G)}(e,\gamma_0) \ge d_{X(G,{\mathbb H},d_G)}(e,a)\ge \frac{2N}{3}. 
\end{eqnarray*}
If $d_{X(G,{\mathbb H},d_G)}(e,a) < \frac{2N}{3}$, then 
\begin{eqnarray*}
d_{X(G,{\mathbb H},d_G)}(a,g_i) \ge d_{X(G,{\mathbb H},d_G)}(e,g_i)-d_{X(G,{\mathbb H},d_G)}(e,a)>\frac{N}{3}
\end{eqnarray*}
for $i=1,2$. 
Hence the claim above implies that 
\begin{eqnarray*}
d_{X(G,{\mathbb H},d_G)}(e,\gamma_0) > \frac{N}{6\mu_0}-\frac{C_0}{\mu_0}-2c_2. 
\end{eqnarray*}
We note that $\frac{2N}{3} > \frac{N}{6\mu_0}-\frac{C_0}{\mu_0}-2c_2$. 

There exists a constant $c_3\ge 0$ such that 
any $(\mu,C)$-quasigeodesic $\gamma$ from $g_1$ to $g_2$ is contained 
in the $c_3$-neighborhood of $\gamma_0$ because $X(G,{\mathbb H},d_G)$ 
is hyperbolic and $\gamma_0$ is a $(\mu_0,C_0)$-quasigeodesic. 
Hence we have 
\begin{eqnarray*}
d_{X(G,{\mathbb H},d_G)}(e,\gamma)>\frac{N}{6\mu_0}-\frac{C_0}{\mu_0}-2c_2-c_3. 
\end{eqnarray*}
Thus we obtain a desired function $f_1$ which is defined by 
$f_1(N)=\max\left\{\frac{N}{6\mu_0}-\frac{C_0}{\mu_0}-2c_2-c_3, 0\right\}$. 
\end{proof}

\section{Proof of Theorem \ref{geomorder}}\label{go}
We fix some notations in this section. 
Let $G$ be a countable group. 
Let $X$ and $Y$ be compact metrizable spaces endowed with geometrically finite convergence
actions of $G$ such that $\H(X)\to \H(Y)$. 
Let $\mathbb K$ and $\mathbb H$ be sets of representatives of conjugacy classes of 
$\H(X)$ and $\H(Y)$, respectively 
and satisfy that for any $K\in \mathbb K$, there exists $H\in \mathbb H$ such that $K\subset H$
(see \cite[Lemma 3.5]{M-O-Y3}). 
The group $G$ is finitely generated relative to $\mathbb K$ by Theorem \ref{s}.
Let $d_G$ be a left invariant proper metric on $G$. 
By rescaling $d_G$ if necessary, we assume that 
$G$ is generated by $S\cup\bigcup_{K\in {\mathbb K}}K$, 
where $S:=\{g\in G \ |\ 0< d_G(e,g)\le 1\}$, and that 
any $H\in\mathbb H$ is generated by 
$(S\cap H)\cup \bigcup_{K\in {\mathbb K}_H}K$, 
where ${\mathbb K}_H:=\{K\in{\mathbb K}\ |\ K\subset H \}$.
Obviously $G$ is also generated by $S\cup\bigcup_{H\in {\mathbb H}}H$. 
We consider the augmented spaces $X(G,{\mathbb H},d_G)$ and $X(G,{\mathbb K},d_G)$ 
and the relative Cayley graphs 
$\overline{\Gamma}(G,{\mathbb H},S)$ and $\overline{\Gamma}(G,{\mathbb K},S)$. 
We have the following commutative diagram:
\[
\begin{CD}
G @>id_G >> G\\
@Vi_{\mathbb K}VV @VVi_{\mathbb H}V \\
X(G,{\mathbb K},d_G) @> >> X(G,{\mathbb H},d_G)\\
@V\pi_{\mathbb K}VV @VV\pi_{\mathbb H}V \\
\overline{\Gamma}(G,{\mathbb K},S) @> >>\overline{\Gamma}(G,{\mathbb H},S).
\end{CD}
\]
We identify 
$i_{\mathbb H}(G)$, $i_{\mathbb K}(G)$, $\pi_{\mathbb H}\circ i_{\mathbb H}(G)$
and $\pi_{\mathbb K}\circ i_{\mathbb K}(G)$ with $G$ 
and also identify 
$i_{\mathbb H}(g)$, $i_{\mathbb K}(g)$, $\pi_{\mathbb H}\circ i_{\mathbb H}(g)$
and $\pi_{\mathbb K}\circ i_{\mathbb K}(g)$ with $g$ 
for every $g\in G$. 

Since $G$ acts on $X(G,{\mathbb K},d_G)$ and $X(G,{\mathbb H},d_G)$ 
properly isometrically,
the induced metrics 
$i_{\mathbb K}^*d_{X(G,{\mathbb K},d_G)}$ and $i_{\mathbb H}^*d_{X(G,{\mathbb H},d_G)}$
on $G$ are left invariant proper metrics on $G$. 
It is well-known that such metrics are coarsely equivalent. 
In particular we have the following:  
\begin{Lem}\label{a'}
There exists a non-decreasing function 
$f_2:[0,\infty)\to [0,\infty)$ 
satisfying $f_2(t)\to \infty$ as $t\to \infty$ such that for any $g_1,g_2\in G$, 
\[
f_2(d_{X(G,{\mathbb K},d_G)}(g_1,g_2))\le 
d_{X(G,{\mathbb H},d_G)}(g_1,g_2).
\]
\end{Lem}

The following is \cite[Proposition 3.9]{Yan11} (see also \cite[Lemma 2.8]{M-O-Y2}): 
\begin{Lem}\label{minilift}
There exist constants $\mu\ge 1$ and $C\ge 0$ 
such that for any geodesic $\overline{\gamma}_\mathbb H$ in $\overline{\Gamma}(G,{\mathbb H},S)$, 
its minimal lift $\overline{\gamma}_\mathbb K$ is 
a locally minimal $(\mu,C)$-quasigeodesic without backtracking 
in $\overline{\Gamma}(G,{\mathbb K},S)$. 
Here a {\it minimal lift} $\overline{\gamma}_\mathbb K$ of $\overline{\gamma}_\mathbb H$ 
is a path given by replacing every $H$-component of $\overline{\gamma}_\mathbb H$ with 
a geodesic in $\overline{\Gamma}(H,{\mathbb K}_H,H\cap S)$ for every $H\in \mathbb H$. 
\end{Lem}

The following is based on the proof of \cite[Lemma 2.1]{Mit98}: 
\begin{Lem}\label{c}
Suppose that for any $N>0$ there exists $M>0$ 
satisfying the following: 
if a geodesic $\gamma_{\mathbb K}$ in $X(G,{\mathbb K},d_G)$ from $g_1\in G$ to $g_2\in G$
satisfies $d_{X(G,{\mathbb K},d_G)}(e, \gamma_{\mathbb K})>M$, 
then every geodesic $\gamma_{\mathbb H}$ in $X(G,{\mathbb H},d_G)$ 
from $g_1$ to $g_2$ satisfies $d_{X(G,{\mathbb H},d_G)}(e,\gamma_{\mathbb H})> N$. 
Then the identity map $id_G:G\to G$ extends to a $G$-equivariant continuous map from 
$G\cup \partial X(G,{\mathbb K},d_G)$ to $G\cup \partial X(G,{\mathbb H},d_G)$. 
\end{Lem}
\begin{proof}
Suppose that the identity map $id_G:G\to G$ cannot extend to 
a $G$-equivariant continuous map from 
$G\cup \partial X(G,{\mathbb K},d_G)$ to $G\cup \partial X(G,{\mathbb H},d_G)$.
Then it follows from Lemma \ref{b} and Remark \ref{b''} that 
we have two sequences $\{g_{n}\}$, $\{g'_{n}\}$, 
a point $q\in \partial X(G,{\mathbb K},d_G)$ and two different points 
$p,p'\in \partial X(G,{\mathbb H},d_G)$
such that $g_{n}\to q$, $g'_{n}\to q$ 
in $X(G,{\mathbb K},d_G)\cup \partial X(G,{\mathbb K},d_G)$, 
and $g_{n}\to p$, $g'_{n}\to p'$ 
in $X(G,{\mathbb H},d_G)\cup \partial X(G,{\mathbb H},d_G)$.
We take a geodesic $\gamma_{{\mathbb K},n}$ in $X(G,{\mathbb K},d_G)$
from $g_{n}$ to $g'_{n}$
and a geodesic $\gamma_{{\mathbb H},n}$ in $X(G,{\mathbb H},d_G)$
from $g_{n}$ to $g'_{n}$
for every $n$. 
Since $g_{n}\to p$, $g'_{n}\to p'$ in $X(G,{\mathbb H},d_G)\cup \partial X(G,{\mathbb H},d_G)$ 
and $p\neq p'$, there exists $N_1>0$ and $N_2>0$ such that for any $n>N_1$, 
$\gamma_{{\mathbb H},n}\cap
\{x\in X(G,{\mathbb H},d_G) \ |\ d_{X(G,{\mathbb H},d_G)}(e,x)\le N_2\}\neq \emptyset$. 
On the other hand, 
since $g_{n}\to q$, $g'_{n}\to q$ in $X(G,{\mathbb K},d_G)\cup \partial X(G,{\mathbb K},d_G)$, 
for any $M>0$, 
there exists $N_3>0$ such that for any $n>N_3$, 
$d_{X(G,{\mathbb K},d_G)}(e,\gamma_{{\mathbb K},n})> M$. 
\end{proof}

\begin{proof}[Proof of Theorem \ref{geomorder}]
For every $(g_1,g_2)\in G\times G$, we fix four paths from $g_1$ to $g_2$: 
the first path which is denoted by $\overline{\gamma_{\mathbb H}}(g_1,g_2)$
is a geodesic in $\overline{\Gamma}(G,{\mathbb H},S)$; 
the second path which is denoted by $\overline{\gamma_{\mathbb K}}(g_1,g_2)$
is a minimal lift of $\overline{\gamma_{\mathbb H}}(g_1,g_2)$ 
in $\overline{\Gamma}(G,{\mathbb K},S)$;
the third path which is denoted by $\gamma_{\mathbb H}(g_1,g_2)$ 
is the typical lift of $\overline{\gamma_{\mathbb H}}(g_1,g_2)$ 
in $X(G,{\mathbb H},d_G)$;
the fourth path which is denoted by $\gamma_{\mathbb K}(g_1,g_2)$ 
is the typical lift of $\overline{\gamma_{\mathbb K}}(g_1,g_2)$ 
in $X(G,{\mathbb K},d_G)$.
We note that 
$G\cap \gamma_{\mathbb H}(g_1,g_2)=G\cap \overline{\gamma_{\mathbb H}}(g_1,g_2)\subset 
G\cap \gamma_{\mathbb K}(g_1,g_2)=G\cap \overline{\gamma_{\mathbb K}}(g_1,g_2)$.
It follows from Lemma \ref{minilift} and Lemma \ref{d}
that we have constants $\mu \ge 1$ and $C\ge 0$ such that 
for any $(g_1, g_2)\in G\times G$, 
$\overline{\gamma_{\mathbb K}}(g_1,g_2)$, $\gamma_{\mathbb H}(g_1,g_2)$ and $\gamma_{\mathbb K}(g_1,g_2)$
are $(\mu,C)$-quasigeodesics in 
$\overline{\Gamma}(G,{\mathbb K},S)$, $X(G,{\mathbb H},d_G)$ and $X(G,{\mathbb K},d_G)$, respectively. 
Since $X(G,{\mathbb H},d_G)$ and $X(G,{\mathbb K},d_G)$ are hyperbolic,  
we have a constant $c\ge 0$ such that 
for any $(g_1,g_2)\in G\times G$, 
$\gamma_{\mathbb H}(g_1,g_2)$ is contained in a $c$-neighborhood of any geodesic from $g_1$ to $g_2$
in $X(G,{\mathbb H},d_G)$
and also $\gamma_{\mathbb K}(g_1,g_2)$ is contained in a $c$-neighborhood of any geodesic from $g_1$ to $g_2$
in $X(G,{\mathbb K},d_G)$. 

We take a function $f_1$ in Lemma \ref{e} and a function $f_2$ in Lemma \ref{a'}.
We consider a function $F:[0,\infty)\ni t\mapsto f_1(f_2(\max\{t-c, 0\}))-c\in \mathbb R$.
Then $F$ is non-decreasing and satisfies $F(t)\to \infty$ as $t\to \infty$. 

We take $t\ge 0$, $(g_1,g_2)\in G\times G$, 
a geodesic $\gamma'_\mathbb K$ from $g_1$ to $g_2$ in $X(G,{\mathbb K},d_G)$
and a geodesic $\gamma'_\mathbb H$ from $g_1$ to $g_2$ in $X(G,{\mathbb H},d_G)$. 
We suppose that $d_{X(G,{\mathbb K},d_G)}(e,\gamma'_\mathbb K)>t$.
Then we have $d_{X(G,{\mathbb K},d_G)}(e,\gamma_\mathbb K(g_1,g_2))>t-c$
and thus $d_{X(G,{\mathbb K},d_G)}(e,G\cap \gamma_\mathbb K(g_1,g_2))>t-c$. 
We have $d_{X(G,{\mathbb H},d_G)}(e,G\cap \gamma_\mathbb H(g_1,g_2))
\ge f_2(d_{X(G,{\mathbb K},d_G)}(e,G\cap \gamma_\mathbb K(g_1,g_2)))\ge f_2(\max\{t-c, 0\})$ 
by Lemma \ref{a'}.
It follows from Lemma \ref{e} that we have 
$d_{X(G,{\mathbb H},d_G)}(e,\gamma_\mathbb H(g_1,g_2))\ge f_1(f_2(\max\{t-c, 0\}))$. 
We have $d_{X(G,{\mathbb H},d_G)}(e,\gamma'_\mathbb H)\ge f_1(f_2(\max\{t-c, 0\}))-c=F(t)$.

For any $N>0$, we can take $M>0$ such that $F(M)> N$ 
because $F(t)\to \infty$ as $t\to \infty$. 
Then $M$ satisfies the condition in Lemma \ref{c}. 
Hence we have a $G$-equivariant continuous map 
from $\partial X(G,{\mathbb K},d_G)$ to $\partial X(G,{\mathbb H},d_G)$. 
Since $\partial X(G,{\mathbb K},d_G)$ is $G$-equivariant homeomorphic to $X$ 
and $\partial X(G,{\mathbb H},d_G)$ is $G$-equivariant homeomorphic to $Y$, 
we have a $G$-equivariant continuous map from $X$ to $Y$. 
\end{proof}

\begin{Cor}\label{pphi}
Let $G$ be a countable group
and $X$ and $Y$ be compact metrizable spaces endowed 
with geometrically finite convergence actions of $G$.
Let $\phi:G\to G$ be an automorphism of $G$.
Suppose that $\phi(\H(X))\to \H(Y)$. 
Then we have a continuous surjective map $\pi:X\to Y$ 
such that $\pi(gx)=\phi(g)\pi(x)$ for any $g\in G$ and any $x\in X$.
Moreover $\phi\cup \pi:G\cup X\to G\cup Y$ is continuous.
\end{Cor}
\begin{proof}
This follows from Theorem \ref{geomorder}, Remark \ref{geomorder'} and Lemma \ref{phi}.
\end{proof}

\section{Applications of Theorem \ref{geomorder}}\label{ct}
In this section we deal with applications of Theorem \ref{geomorder}.

\subsection{Proof of Corollaries \ref{CT} and \ref{CT'''}}
First we prove Corollary \ref{CT}. We need the following: 
\begin{Lem}\label{F_2}
Let $L(0)$ be a finitely generated and virtually non-abelian free group. 
Then there exists a sequence of subgroups $\{L(k)\}_{k\in \N}$
satisfying the following:
\begin{enumerate}
\item[(1)] $L(k)$ is isomorphic to a 
finitely generated and virtually non-abelian free group for any $k\in \N$; 
\item[(2)] $L(k)$ is an almost malnormal and quasiconvex subgroup of $L(k-1)$ for every $k\in \N$;
\item[(3)] $\bigcap_{k\in \N}L(k)$ is a finite subgroup of $L(0)$.
\end{enumerate}
\end{Lem}

\begin{proof}
We fix a finite index free subgroup $M(0)$ of $L(0)$ and put 
$M(0)\setminus \{e\}=\{m_k ~|~ k \in \N\}$. 
We construct $L(k)$ inductively. 
Suppose that $L(k-1)$ is constructed for a positive integer $k \in \N$. 
If $m_{k}$ does not belong to $L(k-1)$, we put $L(k)=L(k-1)$. 
If $m_{k}$ belongs to $L(k-1)$, 
then we have an infinite virtually cyclic subgroup $E(k)$ of $L(k-1)$ which 
contains $m_k$ and satisfies that $L(k-1)$ is hyperbolic relative to $\{E(k)\}$
by \cite[Corollary 1.7]{Osi06b}. 
It follows from Theorem \ref{hypemb-DGO} (or \cite[Theorem 1.2]{M-O-Y4}) 
that there exists a subgroup $L(k)$ of $L(k-1)$ satisfying the following: 
\begin{enumerate}
\setlength{\itemsep}{0mm}
\item[(i)] $L(k)$ is finitely generated, virtually non-abelian free; 
\item[(ii)] $L(k-1)$ is hyperbolic relative to $\{E(k), L(k)\}$. 
\end{enumerate}
By \cite[Theorem 1.5]{Osi06b}, 
$L(k)$ is an almost malnormal and quasiconvex subgroup of $L(k-1)$ for every $k\in \N$. 
By \cite[Proposition 2.36]{Osi06a}, $E(k)\cap L(k)$ is finite and thus 
$m_k$ does not belong to $L(k)$ for every $k \in \N$. 
Hence we have $\left(\bigcap_{k\in \N}L(k)\right) \cap M(0)=\{e\}$. 
Since $M(0)$ is a finite index subgroup of $L(0)$, $\bigcap_{k\in \N}L(k)$ is finite. 
\end{proof}

\begin{proof}[Proof of Corollary \ref{CT}]
By Theorem \ref{hypemb-DGO} (or \cite[Theorem 1.2]{M-O-Y4}), 
$G$ admits a finitely generated and virtually non-abelian free subgroup $L(0)$ 
which is hyperbolically embedded into $G$ relative to $\H(X)$. 
We take a sequence $\{L(n)\}_{n\in\N}$ of subgroups of $L(0)$ described in Lemma \ref{F_2} 
and for each $n\in\N$, we put 
\begin{eqnarray*}
\K_n=\{P\subset L(n-1) ~|~ P=lL(n)l^{-1}\text{ for some }l\in L(n-1)\}.
\end{eqnarray*}
For each $n\in\N$, $\K_n$ is a proper relatively hyperbolic structure on $L(n-1)$ 
and hence there exists a finitely generated and virtually non-abelian free subgroup 
$H(n,n)$ of $L(n-1)$ which is hyperbolically embedded into $L(n-1)$ relative to $\K_n$ 
by Theorem \ref{hypemb-DGO} (or \cite[Theorem 1.2]{M-O-Y4}). 
We take a sequence $\{H(n,n+m)\}_{m\in\N}$ of subgroups of $H(n,n)$ described 
in Lemma \ref{F_2} for each $n\in\N$, that is, 
$\{H(n,n+m)\}_{m\in\N}$ satisfies the following:
\begin{itemize}
\item $H(n,n+m)$ is isomorphic to a 
finitely generated and virtually non-abelian free group for any $m\in \N$; 
\item $H(n,n+m)$ is an almost malnormal and quasiconvex subgroup of $H(n,n+m-1)$ 
for every $m\in \N$;
\item $\bigcap_{m\in \N}H(n,n+m)$ is a finite subgroup of $H(n,n)$.
\end{itemize}

For each $\lambda\in\{0,1\}^\N$ and each $n\in\N$, we put 
\begin{eqnarray*}
I(\lambda,n)=\{i\in\N ~|~ i\le n\text{ and }\lambda(i)=1\}
\end{eqnarray*}
and 
\begin{eqnarray*}
\H^{(n)}_\lambda=\H(X) &\cup& 
\bigcup_{i\in I(\lambda,n)}\{P\subset G ~|~ P=gH(i,n)g^{-1}\text{ for some }g\in G\} \\
&\cup& \{P\subset G ~|~ P=gL(n)g^{-1}\text{ for some }g\in G\}. 
\end{eqnarray*}
We can confirm that $\H^{(n)}_\lambda$ is a proper relatively hyperbolic structure on $G$
by using \cite[Theorem 1.5]{Osi06b}. 
Also we have $\H^{(n)}_\lambda\to\H^{(m)}_\lambda$ if $n \ge m$.
We take a compact metrizable space $X^{(n)}_\lambda$ endowed with 
a geometrically finite convergence action of $G$ such that $\H(X^{(n)}_\lambda)=\H^{(n)}_\lambda$.
By Theorem \ref{geomorder}, $X^{(n)}_\lambda$ is a blow-up of $X^{(m)}_\lambda$ 
if $n \ge m$. We put $X_\lambda=\underleftarrow{\lim}_{n\in \N}X^{(n)}_\lambda$. 

Now we prove the assertions (1) and (2). 
Since we have $\H(X)\to\H(X^{(n)}_\lambda)$ for each $n\in\N$, 
it follows from Theorem \ref{geomorder} that $X^{(n)}_\lambda$ is a blow-down of $X$.  
Hence the induced action of $G$ on $X_\lambda$ is 
a minimal non-elementary convergence action and 
$X_\lambda$ is a blow-down of $X$ such that 
$\H(X_\lambda)=\bigwedge_{n\in\N}\H^{(n)}_\lambda$
by Lemma \ref{diag} and Lemma \ref{inverse}. 
Moreover we have $\H(X_\lambda)=\H(X)$ by the condition (3) in Lemma \ref{F_2}. 

Next we prove the assertion (3). 
Take any $\lambda, \lambda'\in\{0,1\}^\N$. 
If $\lambda^{-1}(\{1\})\supset \lambda'^{-1}(\{1\})$, 
then we can confirm that $X_\lambda$ is a blow-down of $X_{\lambda'}$
in a similar way to the above argument. 
Now we suppose that $\lambda(n)=0$ and $\lambda'(n)=1$ for some $n$. 
Since $H(n,n)$ is strongly quasiconvex relative to $\H_\lambda^{(n)}$ in $G$, 
the limit set $\Lambda(H(n,n),X_{\lambda}^{(n)})$ endowed with the induced action of $H(n,n)$  
consists of conical limit points (see \cite[Theorem 1.2]{Hru10}). 
Since $\Lambda(H(n,n),X_{\lambda})$ is a blow-up of $\Lambda(H(n,n),X_{\lambda}^{(n)})$, 
it follows from Lemma \ref{equivcont} (3)
that the limit set $\Lambda(H(n,n),X_{\lambda})$ endowed with the induced action of $H(n,n)$ 
also consists of conical limit points. 
On the other hand $\Lambda(H(n,n),X_{\lambda'})$ endowed with the induced action of $H(n,n)$
has a point which is not conical. 
Indeed we take the parabolic point $p_m\in X_{\lambda'}^{(n+m)}$ of $H(n,n+m)$ 
for each $m\in \N\cup\{0\}$. 
Then $\{p_m\}_{m\in \N\cup \{0\}}$ defines a unique point $p_{\lambda'}$ of $X_{\lambda'}$.
Since any $p_m$ is not a conical limit point of $X_{\lambda'}^{(n+m)}$
with respect to the induced action of $H(n,n)$, 
it follows from Lemma \ref{inverse} that 
$p_{\lambda'}$ is not a conical limit point of $X_{\lambda'}$
with respect to the induced action of $H(n,n)$.
Also since $p_0$ is a bounded parabolic point of $X_{\lambda'}^{(n)}$ and 
its maximal parabolic subgroup is $H(n,n)$, 
it follows from Lemma \ref{equivcont} (4)
that $p_{\lambda'}$ belongs to $\Lambda(H(n,n),X_{\lambda'})$. 
Assume that $X_{\lambda}$ is a blow-down of $X_{\lambda'}$. 
Then $\Lambda(H(n,n),X_{\lambda})$ is also a blow-down of $\Lambda(H(n,n),X_{\lambda'})$.
In view of Lemma \ref{equivcont} (3), 
this contradicts the facts that  
$\Lambda(H(n,n),X_{\lambda})$ consists of conical limit points and 
that $\Lambda(H(n,n),X_{\lambda'})$ has a point which is not conical.

The assertion (4) follows from the assertion (3).

Since $\H(X)$ is a proper relatively hyperbolic structure on $G$, 
the assertion (5) follows from the assertion (1) and (4). 
\end{proof}

\begin{Rem}\label{CT'}
We consider a characteristic function $\chi_{\{1,\ldots, n\}}\in \{0,1\}^\N$ 
for each $n\in \N$. 
We define $X_{n}$ by $X_{\chi_{\{1,\ldots, n\}}}$ for each $n\in \N$ on the above proof, 
and put $X_0=X$. 
Then $X_n$ is a blow-down of $X_m$ for $n,m\in \N\cup\{0\}$ such that $n\ge m$. 
We can easily confirm that $X_n$ for each $n\in \N\cup\{0\}$ has
exactly $n$ $G$-orbits of points which are neither conical nor parabolic. 
Since any conjugacy between minimal non-elementary convergence actions of $G$ 
preserves conical limit points and parabolic points, 
$X_n$ and $X_m$ for any different $n, m\in \N\cup \{0\}$ are not conjugate. 
\end{Rem}

\begin{Rem}\label{CT''}
Suppose that $G$ is finitely generated and $G$ has a compact metrizable space 
$X$ endowed with a geometrically finite convergence action. 
Since $\Out(G)$ is countable, 
it follows from Corollary \ref{CT} that the set of all conjugacy classes
of minimal non-elementary convergence actions of $G$ whose peripheral
structures are equal to $\H(X)$ is uncountable.
\end{Rem}

\begin{proof}[Proof of Corollary \ref{CT'''}]
This is proved in a similar way in Proof of Corollary \ref{CT}, 
where $\{H(n,n+m)\}_{m\in \N}$ is replaced by 
$\{H(n,n)\}_{m\in \N}$ for each $m,n\in\N$.
We omit details.
\end{proof}

\subsection{Common blow-ups of geometrically finite convergence actions}
Let $\{X_i\}_{i\in I}$ be a family of 
compact metrizable spaces endowed with minimal non-elementary convergence actions
of $G$, where $I$ is a non-empty countable index set.
If the induced action of $G$ on $\bigwedge_{i\in I} X_i$ is a convergence action, 
then $\bigwedge_{i\in I} X_i$ is the smallest common blow-up of all $X_i$, that is, 
$\bigwedge_{i\in I} X_i$ is a common blow-up of all $X_i$ and 
any common blow-up of all $X_i$ 
is a blow-up of $\bigwedge_{i\in I} X_i$ 
(see Lemma \ref{diag}). 
Unfortunately the action of $G$ on $\bigwedge_{i\in I} X_i$ is 
not necessarily a convergence action even if $I$ is finite. 
Indeed a counterexample is given 
by \cite[Theorem 1]{B-R} about 
Cannon-Thurston maps (refer to \cite[Corollary 1.3]{M-O}).
On the other hand, we have the following: 
\begin{Cor}\label{common''}
Let $G$ be a finitely generated group. 
Let $\{X_i\}_{i\in I}$ be a family of 
compact metrizable spaces endowed with geometrically finite convergence actions
of $G$, where $I$ is a non-empty countable index set.
Then the action of $G$ on $\bigwedge_{i\in I} X_i$ is
a minimal non-elementary convergence action of $G$
such that $\H(\bigwedge_{i\in I} X_i)=\bigwedge_{i\in I} \H(X_i)$. 
Here a subgroup of $G$ belongs to $\bigwedge_{i\in I} \H(X_i)$ if and only if 
it is infinite and is an intersection of some $H_i\in \H(X_i)$ for each $i\in I$. 

Moreover if $I$ is finite, then the action of $G$ on $\bigwedge_{i\in I} X_i$ is 
geometrically finite.
\end{Cor}

\begin{Rem}\label{common'''}
If $I$ is infinite, 
then the convergence action of $G$ on $\bigwedge_{i\in I} X_i$
is not necessarily geometrically finite 
(see Remark \ref{Dun} and also Proof of Corollary \ref{CT}). 
\end{Rem}

\begin{Rem}\label{countable}
When $G$ is finitely generated, there exist at most 
countably infinitely many proper relatively hyperbolic structures on $G$ 
(see \cite[Proposition 6.4]{M-O-Y3}).
Hence there exist at most 
countably infinitely many $G$-equivariant homeomorphism classes of 
geometrically finite convergence actions of $G$. 
Thus countability of $I$ is not an essential assumption. 
Also if a countable group $G$ admits a geometrically finite convergence action
on a compact metrizable space, 
$G$ has infinitely many geometrically finite convergence actions
on compact metrizable spaces (see for example \cite[Theorem 6.3]{M-O-Y3}).
\end{Rem}

We need the following, which is a corollary of \cite[Proposition 1.10]{Dah03} 
(see also \cite[Theorem 1.1]{Yan12}): 
\begin{Lem}\label{f}
Let $G$ be a finitely generated group. 
Let $\{X_i\}_{i\in I}$ be a finite family of 
compact metrizable spaces endowed with geometrically finite convergence actions of $G$.
Let $Z$ be a compact metrizable space endowed with a geometrically finite convergence action of $G$
and $\pi_i:Z\to X_i$ be a $G$-equivariant continuous map for every $i\in I$. 
When we take $H_i\in \H(X_i)$ for every $i\in I$, we have 
$\bigcap_{i\in I} \Lambda (H_i,Z)=\Lambda (\bigcap_{i\in I} H_i,Z)$.
\end{Lem}
\begin{proof}
We take $i\in I$ and $H_i\in \H(X_i)$. 
The induced action of $H_i$ on $Z$  
is elementary or geometrically finite (see for example \cite[Theorem 1.1]{Yan11}). 
Also when we take $g\notin H_i$ and the parabolic points $p\in X_i$ of $H_i$ 
and $q\in X_i$ of $gH_ig^{-1}$, we have $p\neq q$. 
Hence we have $\pi_i^{-1}(p)\cap \pi_i^{-1}(q)=\emptyset$ and thus 
$\Lambda(H_i,Z)\cap g\Lambda(H_i, Z)=\emptyset$. 
In particular $H_i$ is fully quasiconvex with respect to $Z$. 
When we use \cite[Proposition 1.10]{Dah03} inductively, we have the assertion. 
\end{proof}

\begin{proof}[Proof of Corollary \ref{common''}]
First we suppose that $I$ is a finite index set.
Since $I$ is a finite set, it follows from \cite[Proposition 7.1]{M-O-Y3} 
based on \cite[Theorem 1.8]{D-S05}
that $\bigwedge_{i\in I}\H(X_i)$ is a proper relatively hyperbolic structure. 
Hence we have a compact metrizable space $Z$ endowed 
with a geometrically finite convergence action such that $\H(Z)=\bigwedge_{i\in I}\H(X_i)$. 
It follows from Theorem \ref{geomorder} that we have 
a $G$-equivariant continuous map $\pi_i: Z\to X_i$ for every $i\in I$.  
The image of $\prod_{i\in I}\pi_i:Z\to \prod_{i\in I}X_i$ is $\bigwedge_{i\in I}X_i$ and    
the action of $G$ on $\bigwedge_{i\in I}X_i$ is a minimal non-elementary convergence action 
such that $\H(\bigwedge_{i\in I}X_i)=\bigwedge_{i\in I}\H(X_i)$ by Lemma \ref{diag}. 
We take $\{r_i\}_{i\in I}\in \bigwedge_{i\in I}X_i$. 
If $r_i\in X_i$ is conical for some $i\in I$, 
then $\{r_i\}_{i\in I}$ is conical and 
thus $(\prod_{i\in I}\pi_i)^{-1}(\{r_i\}_{i\in I})$ 
consists of one point by Lemma \ref{equivcont} (3). 
Suppose that $r_i\in X_i$ is parabolic 
and $H_i\in \H(X_i)$ is the maximal parabolic subgroup for every $i\in I$. 
Since $I$ is a finite set, 
we have 
$\bigcap_{i\in I}\pi_i^{-1}(r_i)= \bigcap_{i\in I} \Lambda (H_i,Z)=\Lambda (\bigcap_{i\in I} H_i,Z)$ 
by Lemma \ref{equivcont} (4) and Lemma \ref{f}.
If $\bigcap_{i\in I} H_i$ is finite, then $\Lambda (\bigcap_{i\in I} H_i,Z)=\emptyset$. 
This contradicts the fact that $\{r_i\}_{i\in I}$ belongs to 
the image of $\prod_{i\in I}\pi_i:Z\to \prod_{i\in I}X_i$.
If $\bigcap_{i\in I} H_i$ is infinite, then $\bigcap_{i\in I} H_i$ is parabolic and thus 
$\Lambda (\bigcap_{i\in I} H_i,Z)$ consists of one point. 
Hence $(\prod_{i\in I}\pi_i)^{-1}(\{r_i\}_{i\in I})$ consists of one point. 
Thus $Z\ni z\mapsto \{\pi_i(z)\}_{i\in I}\in \bigwedge_{i\in I} X_i$ 
is a $G$-equivariant homeomorphism. 

Next we suppose that $I$ is not necessarily finite. 
We consider a countable directed non-empty set $J$ consisting of 
finite subsets of $I$, where $j<j'$ for each $j,j'\in J$ is defined by $j\subset j'$.  
We define $X_j:=\bigwedge_{i\in j}X_i$ for each $j\in J$ 
and the natural $G$-equivariant surjection 
$\pi_{jj'}: X_{j'}\to X_j$ for each $j,j'\in J$ such that $j<j'$. 
Since the action of $G$ on $X_j$ is a geometrically finite convergence action
by the former argument, 
we have a projective system of geometrically finite convergence actions of $G$. 
Since $J$ is countable, $\underleftarrow{\lim}_{j\in J}X_j$ 
is naturally endowed with a minimal non-elementary convergence action of $G$
by Lemma \ref{inverse}. 
We have a $G$-equivariant continuous map 
$\pi_i:\underleftarrow{\lim}_{j\in J}X_j\to X_i$ for any $i\in I$. 
Hence it follows from Lemma \ref{diag} that the image of 
$\prod_{i\in I} \pi_i:\underleftarrow{\lim}_{j\in J}X_j\to \prod_{i\in I} X_i$
is $\bigwedge_{i\in I} X_i$ and 
the action of $G$ on $\bigwedge_{i\in I} X_i$ is a non-elementary convergence action
such that $\H(\bigwedge_{i\in I} X_i)=\bigwedge_{i\in I} \H(X_i)$. 
\end{proof}

\begin{Rem}\label{pf}
On the last part, we can prove that $\underleftarrow{\lim}_{j\in J}X_j$ and 
$\bigwedge_{i\in I}X_i$ are $G$-equivariant homeomorphic. Indeed we have 
the natural projection $\bigwedge_{i\in I}X_i\to X_j$ for any $j\in J$ 
and thus we have a $G$-equivariant continuous map $\bigwedge_{i\in I}X_i\to \underleftarrow{\lim}_{j\in J}X_j$. 
It follows from Lemma \ref{equivcont} (2) that $\underleftarrow{\lim}_{j\in J}X_j$ and 
$\bigwedge_{i\in I}X_i$ are $G$-equivariant homeomorphic.
\end{Rem}

\begin{Rem}
We can prove Corollary \ref{common''} by combining existence of Floyd maps
(\cite[Corollary of Section 1.5]{Ger10}), 
Lemmas \ref{diag} and \ref{inverse}, but we omit details. 
\end{Rem}

We call a geometrically finite convergence action of $G$ on a compact metrizable space $X$ 
the {\it universal geometrically finite convergence action}
if every compact metrizable space endowed with a geometrically finite convergence action of $G$ 
is a blow-down of $X$. 
When $G$ is finitely generated and has a geometrically finite convergence action, 
we obtain the smallest common blow-up of 
all compact metrizable spaces endowed with geometrically finite convergence actions 
by applying Corollary \ref{common''} to a family of 
representatives of $G$-equivariant homeomorphism classes of 
all compact metrizable spaces endowed with geometrically finite convergence actions. 
However the action is not necessarily geometrically finite (see Remark \ref{Dun}). 
On the other hand we recognize many examples of groups with the universal geometrically finite convergence actions. 
Indeed when a convergence action of $G$ on $X$ is geometrically finite, 
it is the universal geometrically finite convergence action if and only if 
$\H(X)$ is the universal relatively hyperbolic structure on $G$ in the sense of 
\cite[Definition 4.1]{M-O-Y3} by Theorem \ref{geomorder} (see also Lemma \ref{equivcont} (5))
and we know many examples of groups with the universal relatively hyperbolic structures 
(see \cite[Remark (II) in Section 7]{M-O-Y3}).
The following is equivalent to an open question 
(see \cite[Question 1.5]{B-D-M09} and also \cite[Remark (IV) in Section 7]{M-O-Y3}):
\begin{Quest}\label{Q} 
Does every finitely presented group with a geometrically finite convergence action 
have the universal geometrically finite convergence action?
Does every torsion-free finitely generated group with a geometrically finite convergence action  
have the universal geometrically finite convergence action?
\end{Quest}

\begin{Rem}\label{Dun}
We consider the smallest common blow-up of all compact metrizable spaces endowed with 
geometrically finite convergence actions of 
the so-called Dunwoody's inaccessible group $J$. 
It follows from \cite[Section 6]{B-D-M09} that $J$
does not have the universal relatively hyperbolic structure 
(see also \cite[Remark (IV) in Section 7]{M-O-Y3}). 
In particular the action is not geometrically finite. 
We note that $J$ is finitely generated but neither finitely presented nor torsion-free.
\end{Rem}

On the other hand the following claims that 
there exists a family of 
compact metrizable spaces endowed with geometrically finite convergence actions
which admits no common blow-downs: 
\begin{Cor}\label{direct}
Let $G$ be a countable group and 
$X_0$ be a compact metrizable space 
endowed with a geometrically finite convergence action of $G$. 
Then we have a sequence of 
compact metrizable spaces endowed with geometrically finite convergence actions $\{X_n\}_{n\in \N}$ 
and a sequence of $G$-equivariant continuous surjections $\{\pi_n:X_{n-1}\to X_n\}_{n\in \N}$
such that $\underrightarrow{\lim}_{n\in \N}X_n$ is not Hausdorff
and there do not exist any common blow-downs of all $X_n$. 
\end{Cor}
\begin{proof}
We consider the set $\{l_k \ |\ k\in \N\}$ of all loxodromic elements of $G$. 
We denote the virtual normalizer $V_G(l_k)$ of $l_k$ by $H_k$. 
Also we denote the conjugacy class represented by $H_k$ by $\H'_k$. 
When we put $\H_0:=\H(X)$, 
$\H_n:=\H(X)\cup \bigcup_{k=1}^n\H'_k$ is also a proper relatively hyperbolic structure
for every $n\in \N$.
In fact if $\H_{n-1}$ is a proper relatively hyperbolic structure, then 
$H_n$ belongs to $\H_{n-1}\setminus \H_0$ 
or is hyperbolically embedded into $G$ relative to $\H_{n-1}$ (\cite[Corollary 1.7]{Osi06b}). 
Hence $\H_n=\H_{n-1}\cup \H'_{n}$ is a proper relatively hyperbolic structure on $G$. 
When we take a compact metrizable space $X_n$ endowed 
with a geometrically finite convergence action such that $\H(X_n)=\H_n$ 
for every $n\in \N$, 
we have a unique $G$-equivariant surjection $\pi_{n}:X_{n-1}\to X_{n}$ for every $n\in \N$ 
by theorem \ref{geomorder} and Remark \ref{geomorder'}.
Then the directed limit $X_\infty:=\underrightarrow{\lim}_{n\in \N}X_n$ is not Hausdorff. 
Indeed assume that $X_\infty$ is Hausdorff. 
Then the action of $G$ on $X_\infty$ is a non-elementary convergence action. 
Hence we have a loxodromic element $l\in G$ with respect to $X_\infty$.
Then $l$ is also loxodromic with respect to $X_0$. 
However $l$ is an element of a parabolic subgroup with respect to $X_n$ for some $n$. 
In particular $l$ has exactly one fixed point in $X_\infty$. 
This contradicts the assumption that $l$ is loxodromic with respect to $X_\infty$. 

Assume that there exists a common blow-down $Y$ of all $X_n$. 
We have a loxodromic element $l\in G$ with respect to $Y$.
Since each $X_n$ is a blow-up of $Y$, 
$l$ is a loxodromic element with respect to each $X_n$. 
This contradicts the fact that $l$ is an element of a parabolic subgroup 
with respect to some $X_n$. 
\end{proof}

\appendix
\def\thesection{Appendix \Alph{section}}
\section{Groups with no non-elementary convergence actions}\label{no}
\def\thesection{\Alph{section}}

There exists a criterion for countable groups to have 
no proper relatively hyperbolic structures 
(see \cite[Theorem 1]{K-N04} together with \cite[Definition 1]{Bow12}, and also \cite[Theorem 2]{A-A-S07}). 
In this appendix we show a similar criterion (Corollary \ref{cri}) 
for countable groups to have no non-elementary convergence actions. 
It follows from the criterion that 
$\SL(n,\Z)\ (n \ge 3)$, the mapping class group of an orientable surface 
of genus $g$ with $p$ punctures, where $3g+p \ge 5$ and so on have 
no non-elementary convergence actions by the same argument in \cite[Section 8]{K-N04}.
We note that a countable group with no subgroups which are isomorphic to 
the free group of rank $2$ has no non-elementary convergence actions (see \cite[Theorem 2U]{Tuk94}).

Let $G$ be a group with a family $\{G_\lambda\}_{\lambda\in\Lambda}$ of subgroups such that 
$G$ is generated by $\bigcup_{\lambda\in\Lambda}G_\lambda$. 
We consider a property $(P)$ defined for groups. 
We define a {\it $(P)$-graph} $\Gamma(\{G_\lambda\}_{\lambda\in\Lambda}, (P))$ 
of $\{G_\lambda\}_{\lambda\in\Lambda}$ 
as follows: the set of vertices is $\{G_\lambda\}_{\lambda\in\Lambda}$ and 
the set of edges is 
$\{G_{\lambda,\mu} | \ G_{\lambda,\mu}\ \text{satisfies (P)}, \lambda,\mu\in\Lambda, \lambda\neq\mu\}$, 
where $G_{\lambda,\mu}$ is a subgroup of $G$ generated by $G_\lambda$ and $G_\mu$. 
An edge $G_{\lambda,\mu}$ has two ends $G_\lambda$ and $G_\mu$. 
Suppose that $S$ is a generating set of $G$. 
Then we consider a family $\{\langle s\rangle\}_{s\in S}$ of subgroups of $G$
and we denote $\Gamma(\{\langle s\rangle\}_{s\in S}, (P))$ by $\Gamma(S, (P))$. 
We note that $\Gamma(S, (P))$ has been studied (see \cite[Introduction]{K-N04}).

If a countable group $G$ is infinite and has no non-elementary convergence actions, 
then we say that $G$ satisfies $(l=1,2)$. 
\begin{Prop}\label{cri'}
Let $G$ be a countable group 
with a family $\{G_\lambda\}_{\lambda\in\Lambda}$ of infinite subgroups such that 
$G$ is generated by $\bigcup_{\lambda\in\Lambda}G_\lambda$. 
If the $(l=1,2)$-graph $\Gamma(\{G_\lambda\}_{\lambda\in\Lambda}, (l=1,2))$ is connected, then 
$G$ satisfies $(l=1,2)$. 
\end{Prop}
\begin{proof}
We consider a convergence action of $G$ on a compact metrizable space $X$. 

We assume that for some $\lambda$ and $\mu$ such that $\lambda\neq\mu$, 
$G_{\lambda,\mu}$ is a parabolic subgroup. 
We denote the parabolic point by $p\in X$.
Then $G_\lambda$ and $G_\mu$ are parabolic subgroups and fix $p$.
For each $\lambda'\in \Lambda$ such that $G_{\lambda,\lambda'}$ satisfies $(l=1,2)$, 
$G_{\lambda,\lambda'}$ cannot have loxodromic elements. 
Hence $G_{\lambda,\lambda'}$ fixes $p$ and thus $G_{\lambda,\lambda'}$ is parabolic.
Then $G_{\lambda'}$ is also parabolic and fixes $p$. 
Since $\Gamma(\{G_\alpha\}_{\alpha\in\Lambda}, (l=1,2))$ is connected, 
for any $\alpha\in \Lambda$, $G_\alpha$ fixes $p$.  
Thus $G$ fixes $p$ since $G$ is generated by $\bigcup_{\alpha\in\Lambda}G_\alpha$.
Since $G$ contains $G_{\lambda,\mu}$ and thus cannot contain any loxodromic elements, 
$G$ is parabolic. 

We assume that for any $\lambda$ and $\mu$ such that $\lambda\neq\mu$, 
$G_{\lambda,\mu}$ is not parabolic.  
We take $G_{\lambda,\mu}$ which is an edge. 
Then $\Lambda(G_{\lambda,\mu},X)$ is a set of exactly two points $r,a\in X$ and thus 
$G_{\lambda,\mu}$ is virtually infinite cyclic. 
Hence $G_{\lambda}$ is virtually infinite cyclic
and has a loxodromic element which fixes $r,a\in X$. 
For each $\lambda'\in \Lambda$ such that $G_{\lambda,\lambda'}$ is an edge, 
$G_{\lambda,\lambda'}$ have a loxodromic element. 
Hence $\Lambda(G_{\lambda,\lambda'},X)$ is a set of exactly two points $r, a\in X$.
For any $\alpha\in \Lambda$, 
$G_\alpha$ fixes $r, a$ since $\Gamma(\{G_\alpha\}_{\alpha\in\Lambda}, (l=1,2))$ is connected.
Thus $G$ fixes $r, a$ since $G$ is generated by $\bigcup_{\lambda\in\Lambda}G_\lambda$. 
Hence $\Lambda(G,X)$ is the set of exactly two points $r,a\in X$.
\end{proof}

We say that 
a group satisfies $(A)$ if it is abelian.
Since infinite abelian groups satisfy $(l=1,2)$, we have the following
(compare with \cite[Theorem 2]{A-A-S07}): 
\begin{Cor}\label{cri}
Let $G$ be a countable group 
with a generating set $S$ of $G$. 
Suppose that any element of $S$ is of infinite order.  
If the $(A)$-graph $\Gamma(S, (A))$ is connected, then 
$G$ satisfies $(l=1,2)$.
\end{Cor}

\def\thesection{Appendix \Alph{section}}
\section{Hyperbolically embedded virtually free subgroups
of relatively hyperbolic groups}\label{sect-DGO}
\def\thesection{\Alph{section}}

In this appendix, we give a remark on hyperbolically embedded subgroups
of relatively hyperbolic groups in the sense of \cite{Osi06b} as a
continuation of \cite{M-O-Y4}.
After the first version of \cite{M-O-Y4} appeared, the notion of a
hyperbolically embedded subgroup was further generalized in
\cite{D-G-O11}.
We can see that 
the argument in the proof of \cite[Theorem 6.14 (c)]{D-G-O11} 
gives an alternative proof of \cite[Theorem 1.2]{M-O-Y4}.
In fact we have a stronger version of \cite[Theorem 1.2]{M-O-Y4}.

\begin{Thm}\label{hypemb-DGO}
Suppose that a group $G$ is not virtually cyclic 
and is hyperbolic relative to a family $\bbk$ of proper subgroups
in the sense of \cite{Osi06a}. 
Then $G$ contains a maximal finite normal subgroup $K(G)$. 
Moreover for every positive integer $m$, 
there exists a subgroup $H$ of $G$ such that 
$H$ is the direct product of a free subgroup of rank $m$ 
and $K(G)$ and that $H$ is hyperbolically embedded into $G$ relative to $\bbk$.
\end{Thm}

We write down the proof for reader's convenience.

\begin{proof}[Proof of Theorem \ref{hypemb-DGO}]
The former assertion follows from \cite[Lemma 3.3]{A-M-O07}. 

We prove the latter assertion. 
For the case where $m=1$, the assertion follows from \cite[Corollaries 1.7 and 4.5]{Osi06b} 
and \cite[Lemma 3.8]{A-M-O07}. 
We prove the assertion for the case where $m=2$. 
For the case where $m>2$, the assertion is proved by changing notation. 

Let $X$ be a finite generating set of $G$ relative to $\bbk$. 
We put $\calk=\bigcup_{K\in \bbk} K\setminus\{1\}$ and $Y=X\cup\calk$. 
By \cite[Corollaries 1.7 and 4.5]{Osi06b} and \cite[Lemma 3.8]{A-M-O07}, 
$G$ contains subgroups $E_1, \ldots, E_6$ of $G$ such that 
$G$ is hyperbolic relative to $\bbk\cup\{E_1, \ldots, E_6\}$ 
and for each $i\in\{1, \ldots, 6\}$, we have $E_i=\langle g_i\rangle \times K(G)$ 
for some element $g_i$ of $G$. 
We put $\cale=\bigcup_{i=1}^6 E_i\setminus\{1\}$. 

For each $i\in\{1, \ldots, 6\}$, we denote by $\hat{d}_i$ 
the relative metric on $E_i$ defined in \cite[Definition 4.2]{D-G-O11} 
by using the Cayley graph $\Gamma(G,Y\cup\cale)$. 
We take a positive integer $n$ such that for each $i\in\{1, \ldots, 6\}$, 
we have $\hat{d}_i(1,g_i^n)>50D$, 
where $D=D(1,0)$ is given by \cite[Proposition 4.14]{D-G-O11} 
applied to the Cayley graph $\Gamma(G,Y\cup\cale)$. 
We put $x=g_1^ng_2^ng_3^n$ and $y=g_4^ng_5^ng_6^n$. 
We denote by $H$ the subgroup of $G$ which is generated by $\{x,y\}$ and $K(G)$. 

Since for each $i\in\{1, \ldots, 6\}$, the subgroup $\langle g_i\rangle$ commutes with $K(G)$, 
we have $H=\langle x,y\rangle\times K(G)$. 
As mentioned in the proof of \cite[Theorem 6.14 (c)]{D-G-O11}, 
the subgroup $\langle x,y\rangle$ is a free group of rank two. 

Let $r=r_1\ldots r_k$ be a path in $\Gamma(G,Y\cup\cale)$ such that 
for each $j \in\{1, \ldots, k\}$, 
the label of the edge $r_j$ belongs to 
$\{(g_1^ng_2^ng_3^n)^{\pm 1}, (g_4^ng_5^ng_6^n)^{\pm 1}\}$. 
Here for each $i\in\{1, \ldots, 6\}$, we think of $g_i^n$ as a letter in $\cale$. 
By \cite[Lemma 4.20]{D-G-O11}, the path $r$ is a $(4,1)$-quasigeodesic. 
Now we consider the metrics on $\langle x,y\rangle$, $H$ and $G$ 
induced by the Cayley graphs $\Gamma(\langle x,y\rangle,\{x,y\})$, 
$\Gamma(H,\{x,y\}\cup K(G))$ and $\Gamma(G,Y\cup\cale)$, respectively.
Then the natural embedding from $\langle x,y\rangle$ into $G$ is quasi-isometric. 
Since $H$ contains $\langle x,y\rangle$ as a finite index subgroup,  
the natural embedding from $H$ into $G$ is also quasi-isometric.

By the proof of \cite[Theorem 6.14 (c)]{D-G-O11}, 
the subgroup $H$ is hyperbolically embedded in $G$ in the sense of \cite{D-G-O11}. 
Hence $H$ is almost malnormal in $G$ by \cite[Proposition 2.8]{D-G-O11}. 

Therefore $G$ is hyperbolic relative to $\bbk\cup\{E_1, \ldots, E_6\}\cup\{H\}$ 
by \cite[Theorem 1.5]{Osi06b}.
Since for each $i\in\{1, \ldots, 6\}$, the subgroup $E_i$ is a hyperbolic group, 
the group $G$ is hyperbolic relative to $\bbk\cup\{H\}$ by \cite[Theorem 2.40]{Osi06a}.
\end{proof}



\begin{thebibliography}{100}
\bibitem{A-A-S07}
J. W. Anderson, J. Aramayona, and K. J. Shackleton, 
An obstruction to the strong relative hyperbolicity of a group, 
J. Group Theory 10 (2007), no. 6, 749-756. 

\bibitem{A-M-O07}
G. Arzhantseva, A. Minasyan and D. Osin, 
The SQ-universality and residual properties of relatively hyperbolic groups, 
J. Algebra 315 (2007), no. 1, 165--177.

\bibitem{B-R}
O. Baker and T. Riley, 
Cannon-Thurston maps do not always exist,
preprint, arXiv:1206.0505v3.

\bibitem{B-D-M09}
J. Behrstock, C. Dru\c{t}u and L. Mosher, 
Thick metric spaces, relative hyperbolicity, and quasi-isometric rigidity, 
Math. Ann. 344 (2009), no. 3, 543--595.

\bibitem{Bow99b}
B. H. Bowditch, Convergence groups and configuration spaces, 
Geometric group theory down under (Canberra, 1996), 23--54, 
de Gruyter, Berlin, 1999.

\bibitem{Bow12}
B. H. Bowditch, Relatively hyperbolic groups, 
Internat. J. Algebra. Comput. 22, no. 3 (2012).

\bibitem{Dah03}
F. Dahmani, 
Combination of convergence groups, 
Geom. Topol. 7 (2003), 933--963.

\bibitem{D-G-O11}
F. Dahmani, V. Guirardel and D. Osin, 
Hyperbolically embedded subgroups and rotating families in groups acting on hyperbolic spaces, 
preprint, arXiv:1111.7048v3.

\bibitem{D-S05}
C. Dru\c{t}u and M. Sapir, 
Tree-graded spaces and asymptotic cones of groups. 
With an appendix by D. Osin and M. Sapir, 
Topology 44 (2005), no. 5, 959--1058.

\bibitem{Far98}
B. Farb,
Relatively hyperbolic groups, 
Geom. Funct. Anal. 8 (1998), no. 5, 810--840. 

\bibitem{Flo80}
W. J. Floyd, 
Group completions and limit sets of Kleinian groups, 
Invent. Math. 57 (1980), no. 3, 205--218.

\bibitem{Fre95}
E. M. Freden, Negatively curved groups have the convergence property I, 
Ann. Acad. Sci. Fenn. Ser. A I Math. 20 (1995), no. 2, 333--348.


\bibitem{G-M87}
F. W. Gehring and G. J. Martin, 
Discrete quasiconformal groups I, 
Proc. London Math. Soc. (3) 55 (1987), no. 2, 331--358.

\bibitem{Ger09}
V. Gerasimov, 
Expansive convergence groups are relatively hyperbolic, 
Geom. Funct. Anal. 19 (2009), no. 1, 137--169.

\bibitem{Ger10}
V. Gerasimov, 
Floyd maps for relatively hyperbolic groups. 
Geom. Funct. Anal. 22 (2012), no. 5, 1361--1399.

\bibitem{G-P09}
V. Gerasimov and L. Potyagailo,
Quasi-isometric maps and Floyd boundaries of relatively hyperbolic groups, 
preprint, arXiv:0908.0705v8. To appear in JEMS.

\bibitem{G-P13}
V. Gerasimov and L. Potyagailo,
Similar relatively hyperbolic actions of a group, 
preprint, arXiv:1305.6649. 

\bibitem{Gro87}
M. Gromov, 
Hyperbolic groups, 
Essays in group theory (S. Gersten, ed.), 75--263, MSRI Publications 8, 
Springer-Verlag, 1987.

\bibitem{G-M08}
D. Groves and J. F. Manning, 
Dehn filling in relatively hyperbolic groups. 
Israel J. Math. 168 (2008), 317--429.

\bibitem{Hru10}
G. C. Hruska,
Relative hyperbolicity and relative quasiconvexity for countable groups, 
Algebr. Geom. Topol. 10 (2010), no. 3, 1807--1856. 

\bibitem{K-N04}
A. Karlsson and G. A. Noskov, 
Some groups having only elementary actions on metric spaces 
with hyperbolic boundaries, 
Geom. Dedicata 104 (2004), 119--137.

\bibitem{M-O}
Y. Matsuda and S. Oguni,
On Cannon-Thurston maps for relatively hyperbolic groups, 
preprint, arXiv:1206.5868. To appear in J. Group Theory.

\bibitem{M-O-Y3}
Y. Matsuda, S. Oguni and S. Yamagata,
The universal relatively hyperbolic structure on a group
and relative quasiconvexity for subgroups, 
preprint, arXiv:1106.5288v3. 

\bibitem{M-O-Y4}
Y. Matsuda, S. Oguni and S. Yamagata,
Hyperbolically embedded virtually free subgroups of relatively hyperbolic groups, 
preprint, arXiv:1109.2663v3.

\bibitem{M-O-Y2}
Y. Matsuda, S. Oguni and S. Yamagata,
On relative hyperbolicity for a group and relative quasiconvexity for a subgroup, 
preprint, arXiv:1301.4029.

\bibitem{Mit98}
M. Mitra, 
Cannon-Thurston maps for trees of hyperbolic metric spaces. 
J. Differential Geom. 48 (1998), no. 1, 135--164.

\bibitem{Mj09}
M. Mj, 
Cannon-Thurston maps for pared manifolds of bounded geometry. 
Geom. Topol. 13 (2009), no. 1, 189--245. 

\bibitem{Osi06a}
D. Osin, 
Relatively hyperbolic groups: intrinsic geometry, 
algebraic properties, and algorithmic problems, 
Mem. Amer. Math. Soc. 179 (2006), no. 843  

\bibitem{Osi06b}
D. Osin, 
Elementary subgroups of relatively hyperbolic groups and bounded generation, 
Internat. J. Algebra. Comput. 16, no. 1 (2006), 99--118. 

\bibitem{Tuk88}
P. Tukia, 
A remark on a paper by Floyd, 
Holomorphic functions and moduli, Vol. II (Berkeley, CA, 1986), 165--172, 
Math. Sci. Res. Inst. Publ., 11, Springer, New York, 1988. 

\bibitem{Tuk94}
P. Tukia, 
Convergence groups and Gromov's metric hyperbolic spaces, 
New Zealand J. Math. 23 (1994), no. 2, 157--187.

\bibitem{Tuk98}
P. Tukia, Conical limit points and uniform convergence groups, 
J. Reine Angew. Math. 501 (1998), 71--98.

\bibitem{Yam04}
A. Yaman, 
A topological characterisation of relatively hyperbolic groups, 
J. Reine Angew. Math. 566 (2004), 41--89.

\bibitem{Yan12}
W. Yang,
Limit sets of relatively hyperbolic groups,
Geom. Dedicata 156 (2012), 1--12.

\bibitem{Yan11}
W. Yang, Peripheral structures of relatively hyperbolic groups,
published online in J. Reine Angew. Math. 
\end{thebibliography}
\end{document}